\documentclass[letterpaper,11pt]{article}

\usepackage{amssymb, amsmath, amsthm, amsfonts, mathrsfs}
\usepackage[colorlinks=true,linkcolor=blue,anchorcolor=blue,citecolor=black,urlcolor=blue,plainpages=false,pdfpagelabels]{hyperref}
\usepackage{color}
\usepackage{fullpage}
\usepackage{graphicx}
\usepackage{booktabs} 
\usepackage[font=small,labelfont=bf]{caption}
\usepackage{subcaption}
\usepackage{mathtools}
\usepackage{comment}
\usepackage{todonotes}
\usepackage{wrapfig}
\usepackage{placeins}
\usepackage{needspace}\usepackage{titlesec}

\titleformat{\paragraph}[runin]
  {\normalfont\normalsize\bfseries}
  {}
  {0pt}
  {}
  [.]

\usepackage[capitalise]{cleveref}
\crefname{figure}{Figure}{\textbf{Figures}}
\crefname{section}{Section}{\textbf{Sections}}
\crefname{appendix}{Appendix}{\textbf{Appendices}}
\crefname{equation}{Equation}{\textbf{Equations}}
\crefname{remark}{Remark}{\textbf{Remarks}}
\crefname{theorem}{Theorem}{\textbf{Theorems}}
\crefname{definition}{Definition}{\textbf{Definitions}}
\crefname{lemma}{Lemma}{\textbf{Lemmas}}
\crefname{proposition}{Proposition}{\textbf{Propositions}}
\crefname{corollary}{Corollary}{\textbf{Corollaries}}
\crefname{table}{Table}{\textbf{Tables}}

\usepackage{tikz}
\usepackage{tikz-cd}
\usetikzlibrary{arrows}
\tikzset{>=latex}
\tikzset{node distance=50pt, auto}
\usetikzlibrary{decorations.pathreplacing,angles,quotes}
\usepackage[all]{xy}
\usepackage{multirow,booktabs}

\newtheorem{theorem}{Theorem}[section]
\newtheorem{lemma}[theorem]{Lemma}

\theoremstyle{definition}
\newtheorem{definition}[theorem]{Definition}
\theoremstyle{remark}

\newcommand{\VR}{\mathrm{VR}}

\def\noteson{\gdef\ines##1{\noindent{\color{blue}[Inés: ##1]}}}
\noteson

\DeclareMathAlphabet{\mathpzc}{OT1}{pzc}{m}{it}
\usepackage[mathscr]{euscript}

\makeatletter
\newcommand\DEFINEALPHABETLOOP[3]{%
  \ifx\relax#3\expandafter\@gobble\else\expandafter\@firstofone\fi
  {\expandafter\newcommand\expandafter*\csname#3#1\endcsname{#2{#3}}%
   \DEFINEALPHABETLOOP{#1}{#2}}%
}%
\newcommand\Definealphabet[2]{%
  \DEFINEALPHABETLOOP{#1}{#2}abcdefghijklmnopqrstuvwxyzABCDEFGHIJKLMNOPQRSTUVWXYZ\relax
}%
\makeatother
\Definealphabet{bb}{\mathbb}
\Definealphabet{cal}{\mathcal}
\Definealphabet{bf}{\mathbf}
\Definealphabet{sf}{\mathsf}
\Definealphabet{frak}{\mathfrak}
\Definealphabet{scr}{\mathscr}

\usepackage{authblk}

  \usepackage[
    backend=bibtex,
    style=numeric,
    backref=true
  ]{biblatex}

\begin{document}
\title{Density-Robust Spherical Coordinates from Persistent Cohomology}

\author[1]{Nick Nordwald$^{*}$}
\author[2]{Inés García-Redondo$^{*}$}
\author[3]{Anthea Monod}

\affil[1]{Department of Mathematics, University of Bonn, Bonn, Germany}
\affil[2]{Department of Informatics, University of Fribourg, Fribourg, Switzerland}
\affil[3]{Department of Mathematics, Imperial College London, London, UK}

\renewcommand{\thefootnote}{\fnsymbol{footnote}}
\footnotetext[1]{Equal contribution.}
\renewcommand{\thefootnote}{\arabic{footnote}}
\date{}

\maketitle            
\begin{abstract}
Persistent cohomology provides a principled framework for constructing nonlinear coordinates that reflect the topology of data. However, these topological coordinates can be severely distorted by non-uniform sampling density, limiting their applicability to real-world data.  While density-robust circular coordinates have recently been developed, the extension to spherical coordinates remains an open challenge: unlike the circular case, spherical coordinates are obtained through a nonlinear variational problem for sphere-valued maps, to which existing density-correction mechanisms are not directly applicable. In this paper, we introduce the first density-robust construction of spherical coordinates from persistent cohomology. Rather than modifying the coordinate optimization itself, we extend a subsampling-and-alignment framework for circular coordinates to $S^2$, which first computes spherical coordinates on approximately uniform subsamples obtained by rejection sampling and then combines them into a global consensus map.  The principal mathematical difficulty is the alignment of independently computed sphere-valued coordinates.  We formulate this challenge as a spherical Procrustes problem and establish approximation guarantees for a computationally tractable Euclidean relaxation.  Our resulting construction is robust to non-uniform sampling and retains the accuracy of classical spherical coordinates under uniform sampling.  Moreover, by computing persistent cohomology only on fixed-size subsamples, our approach avoids the quartic memory bottleneck of the classical spherical coordinate pipeline and scales to substantially larger datasets.  We conduct experiments on synthetic data and demonstrate accurate coordinate recovery under severe sampling bias and scalability to datasets of 10,000 points.
\end{abstract}

\section{Introduction}
\label{sec:intro}
A common assumption in modern data analysis and machine learning is that high-dimensional data lie near a low-dimensional manifold embedded in a high-dimensional ambient space~\cite{roweis2000nonlinear,maaten2008visualizing}. Classical and modern dimensionality reduction techniques, including nonlinear methods such as Isomap, t-SNE, and UMAP, construct Euclidean embeddings that aim to preserve local or global geometric structure~\cite{tenenbaum2000global,maaten2008visualizing,mcinnes2018umap}. Although highly successful, these methods implicitly assume that the underlying manifold is well represented by a Euclidean coordinate chart. This assumption breaks down when the latent space possesses non-trivial topology. For example, data generated by periodic phenomena, directional measurements, or surfaces enclosing three-dimensional objects are naturally parameterized by circles or spheres rather than Euclidean domains. In such settings, Euclidean embeddings inevitably introduce distortion or redundancy, which obscures the intrinsic structure of the data. Topological methods provide an alternative approach by recovering coordinates that respect the topology of the underlying space rather than enforcing a linear representation~\cite{carlsson2009topology,ghrist2008barcodes}. In particular, persistent cohomology together with the Brown representability theorem gives rise to \emph{topological coordinates}: degree-one cohomology classes allow for circle-valued coordinates~\cite{desilva2011circular}, while degree-two classes yield sphere-valued coordinates~\cite{schonsheck2024spherical,hatcher2002algebraic}. These constructions recover intrinsically nonlinear parameterizations directly from the topology of the data, without requiring a predefined generative model or global Euclidean coordinate chart.

A known limitation of these topological coordinates is their sensitivity to uneven sampling. Current methods to make such coordinates robust to non-uniform sampling have been confined to degree one~\cite{paik2023circular,blumberg2024subsampling}. The broader question of whether density-robust topological coordinates can be extended to higher-degree persistent cohomology therefore remains open, with the degree-two case being the first natural extension to study. Additionally, from the perspective of applications, sphere-valued latent variables arise in settings involving directional or orientation data, surfaces surrounding three-dimensional objects, and other phenomena whose intrinsic parameter space has spherical rather than circular topology. In such applications, uneven sampling often reflects the observation process rather than the geometry of the latent space, so the recovered coordinates should capture the underlying topology without being distorted by the density of the measurements. However, extending density robustness from degree one to degree two is not simply a matter of straightforwardly replacing $S^1$ by $S^2$, since the circular and spherical constructions have fundamentally different analytic structures. Circular coordinates are obtained by selecting harmonic real-valued cocycle representatives within a fixed integral cohomology class. This linear formulation permits density corrections to be introduced through weighted harmonic energies, as proposed by Paik and Park~\cite{paik2023circular}. Spherical coordinates, in turn, are constructed through a nonlinear variational problem for sphere-valued maps; no analogous harmonic-cocycle representation is available. Consequently, the weighting mechanism used in the degree-one setting does not directly extend to degree two.

In this paper, we introduce a density-robust construction of spherical coordinates from degree-two persistent cohomology using an alternative strategy. Rather than modifying the spherical-coordinate optimization, we extend an existing subsampling-and-alignment framework~\cite{blumberg2024subsampling} from $S^1$ to $S^2$. The strategy is to use rejection sampling to construct approximately uniform subsamples, compute a spherical coordinate on each subsample, and combine the resulting maps into a single consensus coordinate. The main mathematical difficulty is that independently computed sphere-valued coordinates are defined only up to a global rotation and therefore cannot be averaged directly. We address this challenge by formulating a spherical Procrustes problem and proving an approximation guarantee for a computationally tractable Euclidean relaxation.

\paragraph{Contributions} Our specific contributions are twofold. First, we extend density-robust topological coordinates from degree-one to degree-two persistent cohomology by introducing the first density-robust construction of spherical coordinates. Our approach extends the subsampling-and-alignment framework of Blumberg, Carri\`{e}re, Fung, and Mandell~\cite{blumberg2024subsampling} beyond the circular setting. The principal new mathematical component is a spherical Procrustes formulation for aligning independently computed sphere-valued coordinates, together with approximation guarantees for an Euclidean relaxation with lower computational cost. Second, we demonstrate that our resulting construction is both robust and practical. Experimentally, it recovers accurate spherical coordinates under substantial non-uniform sampling while matching the classical construction on uniformly sampled data. Combined with two implementation improvements, our proposed framework also substantially improves the practical scalability of spherical coordinates by avoiding the quartic memory bottleneck associated with applying the classical construction directly to the full dataset. This enables computations on substantially larger datasets. An implementation of our approach and scripts to reproduce all experiments in this paper are publicly available at: \texttt{\url{https://github.com/NickNordwald/SphericalCoordinates}}.

\paragraph{Outline} The remainder of this paper is structured as follows. Section~\ref{sec:background} reviews the background on persistent cohomology, Brown representability, topological coordinates, and existing density-robust circular coordinates. Section~\ref{sec:density_robust_spherical} introduces our density-robust construction of spherical topological coordinates, including the subsampling framework, the spherical Procrustes problem, and our approximation theorem for its Euclidean relaxation. Section~\ref{sec:experiments} evaluates the proposed
method experimentally, demonstrating both its computational scalability and its robustness to non-uniform sampling. Finally, Section~\ref{sec:conclusion} concludes with a discussion of limitations and
directions for future work.

\section{Background}
\label{sec:background}
In this section, we present the background relevant to this work. We begin by introducing persistent (co)homology (\cref{sec:ph}), a methodology for extracting multiscale of topological features from data. This theory, in conjunction with the Brown representability theorem (\cref{sec:brown}), is the basis from which circular coordinates are derived following \cite{desilva2011circular} (\cref{sec:circular_coords}), as well as spherical coordinates following \cite{schonsheck2024spherical} (\cref{sec:spherical_coords}). We close the section with an overview of the approach by \cite{blumberg2024subsampling}, which makes the circular coordinate construction robust to non-uniformities in the sampling (\cref{sec:density_robust_circular_coords}).

\subsection{Persistent (Co)homology}
\label{sec:ph}

We first review the persistent homology and cohomology pipelines; see
\cite{zomorodian2004computing, edelsbrunner2008persistent, carlsson2009topology} for further details. 

Considering that our input data comes in the form of a finite metric space $(X,d)$, we begin by constructing a \emph{filtration}; i.e., a
nested family of simplicial complexes indexed by a scale parameter $\varepsilon$. A standard choice, especially in applications for computational reasons, is the
\emph{Vietoris--Rips filtration}, which comprises of Vietoris--Rips complexes at scale $\varepsilon \geq 0$, each given by
all finite subsets of $X$ of diameter at most $\epsilon$,
\[
\VR_\varepsilon(X) \coloneqq \{\sigma \subset X : d(x,y) \leq \epsilon
\ \text{ for all } x, y \in \sigma\}.
\]
Since $\VR_\varepsilon(X) \subseteq \VR_{\varepsilon'}(X)$ whenever
$\varepsilon \leq \varepsilon'$, these complexes form an increasing family, as desired.
 
Next, we use homology to track how the topology of this family changes with $\epsilon$. For a simplicial complex $K$ and an abelian group $G$, let
$C_k(K;G)$ denote the group of $k$-chains---formal $G$-linear combinations
of oriented $k$-simplices---and
$C^k(K;G) \coloneqq \mathrm{Hom}(C_k(K;\Zbb), G)$ denote the group of
$k$-cochains, that is, the $G$-valued functions on $k$-simplices. The boundary
and coboundary operators
\[
\partial_k \colon C_k(K;G) \to C_{k-1}(K;G)
\qquad
\delta^k \colon C^k(K;G) \to C^{k+1}(K;G),
\]
respectively, make these into a chain and a cochain complex. Their kernels are the
\emph{cycles} $Z_k(K;G) \coloneqq \ker \partial_k$ and \emph{cocycles}
$Z^k(K;G) \coloneqq \ker \delta^k$, and their images are the \emph{boundaries}
$B_k(K;G) \coloneqq \mathrm{im}\,\partial_{k+1}$ and \emph{coboundaries}
$B^k(K;G) \coloneqq \mathrm{im}\,\delta^{k-1}$. The \emph{homology} and
\emph{cohomology} groups are the quotients
\[
H_k(K;G) \coloneqq Z_k(K;G)/B_k(K;G),
\qquad
H^k(K;G) \coloneqq Z^k(K;G)/B^k(K;G),
\]
respectively; each are functorial in $K$.
 
Fixing a field $G = \Bbbk$ and applying $H_k(-;\Bbbk)$ in a fixed degree $k$ to
the filtration yields the \emph{persistent homology module}, which is a family of vector
spaces $\{H_k(\VR_\epsilon(X);\Bbbk) : \epsilon \geq 0\}$ together with linear
maps $H_k(\VR_\epsilon(X);\Bbbk) \to H_k(\VR_{\epsilon'}(X);\Bbbk)$ induced by the
inclusions for $\epsilon \leq \epsilon'$. In practice, we usually take $\Bbbk = \Fbb_p$ or
$\Bbbk = \Rbb$. By the Krull--Remak--Schmidt--Azumaya theorem, also referred to as the Structure Theorem in the persistence literature
\cite{azumaya1950corrections, crawley-boevey2015decomposition, botnan2020decomposition}, the persistent homology module decomposes uniquely as a direct sum of \emph{interval modules},
each equal to $\Bbbk$ on some interval $[b,d)$ and $0$ elsewhere, with identity
internal maps across the interval.
 
The multiset of intervals in the module decomposition is a complete invariant of the
module, called its \emph{persistence barcode}. Each bar $[b,d)$ records a topological
feature that appears at the \emph{birth} scale $b$ and vanishes at the
\emph{death} scale $d$; long bars, i.e.,~bars with high persistence $p=d - b$, are typically interpreted as ``signal'' features of
the data, whereas short ones are typically interpreted as noise.
 
The same construction applies directly to cohomology. Applying $H^k(-;\Bbbk)$
gives the \emph{persistent cohomology module}
$\{H^k(\VR_\epsilon(X);\Bbbk) : \epsilon \geq 0\}$, whose internal maps
$H^k(\VR_{\epsilon'}(X);\Bbbk) \to H^k(\VR_\epsilon(X);\Bbbk)$ run in the opposite
direction. The Structure Theorem still applies and it turns out that the persistent homology and
cohomology barcodes coincide \cite{silva2011dualities}. In this paper, the coordinate constructions discussed and presented make use of this cohomological barcode.
 
\subsection{Brown Representability Theorem}
\label{sec:brown}
 
The construction of degree-one and degree-two coordinates relies on a classical
result from homotopy theory, the \emph{Brown representability theorem}, which we now recall.
 
For an abelian group $G$ and an integer $n \geq 1$, a pointed space $K(G,n)$ is
an \emph{Eilenberg--MacLane space} if its homotopy groups are
\[
\pi_k(K(G,n)) \cong
\begin{cases}
G & \text{if } k = n,\\
0 & \text{otherwise.}
\end{cases}
\]
Such a space exists for every pair $(G,n)$ and is unique up to homotopy
equivalence. Brown representability identifies Eilenberg--MacLane spaces as classifying spaces
for cohomology in the following way. Recall that a CW-complex is a topological space built by attaching cells of increasing dimension, providing the natural setting for classical homotopy theory. For every CW-complex $\Xcal$, there is a natural bijection
\[
[\Xcal,\, K(G,n)] \cong H^n(\Xcal; G)
\]
between homotopy classes of maps $\Xcal \to K(G,n)$ and degree-$n$ cohomology
classes of $\Xcal$ with coefficients in $G$.

The relevance of Brown representability to degree-one and degree-two coordinates arises from the low-degree
Eilenberg--MacLane spaces. The circle is a $K(\Zbb,1)$, since
$\pi_1(S^1) = \Zbb$ and all higher homotopy groups vanish, so
\(
[\Xcal, S^1] \cong H^1(\Xcal; \Zbb).
\)
A circular coordinate on the data is precisely a map to $S^1$ and therefore, a degree-one
integral cohomology class of $\Xcal$ yields one
\cite{desilva2011circular}; see \cref{sec:circular_coords} for further details.
In a similar manner, spherical coordinates require maps to $S^2$, however, $S^2$ is not an
Eilenberg--MacLane space. Indeed, $S^2$ has $\pi_2(S^2)=\Zbb$ besides nontrivial homotopy groups in dimensions higher than two (for example, $\pi_3(S^2) \cong \mathbb{Z}$) unlike an Eilenberg--MacLane space which only has a single nontrivial homotopy group. The relevant Eilenberg--MacLane space is instead
$K(\Zbb,2) \simeq \Cbb P^\infty$, with the homotopy equivalence following from the
fibration $S^1 \to S^\infty \to \Cbb P^\infty$ together with the contractibility
of $S^\infty$. In \cref{sec:spherical_coords}, we will recall how the canonical
inclusion $S^2 \hookrightarrow \Cbb P^\infty$ is used to pass from a degree-two
cohomology class to a sphere-valued coordinate \cite{schonsheck2024spherical}.

\subsection{Circular Topological Coordinates}
\label{sec:circular_coords}
We now review the construction by de Silva et al.~\cite{desilva2011circular} which builds circular topological coordinates from data.
Let $(X,d)$ be a finite metric space. We consider the Vietoris--Rips filtration
\(
\VR_\bullet (X) = \{\mathrm{VR}_\varepsilon(X),\, \varepsilon \ge 0\}
\)
and compute persistent cohomology with coefficients in a finite field
$\mathbb{F}_p$. Suppose that the resulting persistence barcode in degree one
contains a prominent interval, and let
\[
[\alpha_p] \in H^1(\mathrm{VR}_\varepsilon(X);\mathbb{F}_p)
\]
denote its representative class at a fixed scale $\varepsilon$ chosen within the
lifetime of this feature.

\paragraph{Lifting to integer coefficients}
To realize this class as a degree-one coordinate, we fix the simplicial complex
$\mathrm{VR}_\varepsilon(X)$ and construct a map into $S^1$.
Restricting this map to the $0$-simplices yields the map on the data.
By the Brown representability theorem, homotopy classes of maps
\(
\theta\colon \mathrm{VR}_\varepsilon(X) \to S^1
\)
are in natural correspondence with elements of
$H^1(\mathrm{VR}_\varepsilon(X);\mathbb{Z})$, so an integral cocycle representative associated to $[\alpha_p]$ is first required.

In order to find such a representative, choose a cocycle
\(\alpha_p \in Z^1(\mathrm{VR}_\varepsilon(X);\mathbb{F}_p)\)
representing the class $[\alpha_p]$ and define an integer-valued cochain
\[
\alpha \in C^1(\mathrm{VR}_\varepsilon(X);\mathbb{Z})
\]
by replacing each coefficient with its canonical representative in
$\{-(p-1)/2,\dots,(p-1)/2\} \subset \mathbb{Z}$.
The obstruction to lifting $[\alpha_p]$ to integral coefficients is its image under
the Bockstein homomorphism of the sequence
$0 \to \mathbb{Z} \xrightarrow{\cdot p} \mathbb{Z} \to \mathbb{F}_p \to 0$,
and hence lies in the $p$-torsion subgroup of
$H^2(\mathrm{VR}_\varepsilon(X);\mathbb{Z})$. Since $\mathrm{VR}_\varepsilon(X)$ is a
finite simplicial complex, its integral cohomology is finitely generated and thus has
torsion for at most finitely many primes, so the obstruction vanishes for all but
finitely many $p$. In this case a lift exists, although the cochain above need not
itself be a cocycle and may first require correction by a multiple of a coboundary;
since this correction can be computationally expensive and the situation arises only
rarely, it is possible to instead simply repeat the computation over a different prime. This gives rise to an integral cocycle $\alpha$ defining a class in
$H^1(\mathrm{VR}_\varepsilon(X);\mathbb{Z})$.

\paragraph{Cocycle integration}
Although integral cocycles suffice for the homotopy-theoretic classification,
they are not suitable for defining an informative degree-one coordinate on the data in the following sense. Given an integral $1$-cocycle
\[
\alpha \in Z^1(\mathrm{VR}_\varepsilon(X);\mathbb{Z}),
\]
we may construct a continuous map
\[
\theta\colon \mathrm{VR}_\varepsilon(X) \longrightarrow S^1
\]
by sending all vertices to a fixed basepoint of $S^1$ and mapping each oriented
edge $e$ to a loop that winds $\alpha(e)$ times around the circle, extended
linearly over higher-dimensional simplices. This procedure realizes the cohomology class
$[\alpha]$ under the correspondence provided by Brown representability. However, when restricted to the $0$-simplices, i.e., to the data set $X$, this
map assigns the same value to every vertex. Consequently, it does not yield a
nontrivial coordinate function on the data. One solution to this problem is to
consider cocycles with real coefficients that are cohomologous to $\alpha$ and
build the map from them.

Let $\tilde{\alpha}$ be a real-valued cocycle cohomologous to $\alpha$, that is, there exist $f \in C^0(\VR_\epsilon(X); \Rbb)$ such that
\[
\tilde{\alpha} = \alpha + \delta^0 f.
\]
A circle-valued function on the data is constructed inductively over the skeleta of $\VR_\epsilon(X)$ as follows. For the vertices, define a map from the real-valued $0$-cochain $f$ by setting
\[
\theta(v) := f(v) \bmod \mathbb{Z}, \qquad v \in X.
\]
For each edge $e = [v,w]$, the relation
\[
\theta(w) - \theta(v) \equiv f(w) - f(v) = \tilde{\alpha}(e) - \alpha(e) \equiv \tilde{\alpha}(e) \pmod{\mathbb{Z}}
\]
ensures that the map $\theta$ can be extended over $e$ as an affine map whose total signed variation along the edge is
$\tilde{\alpha}(e)$. The cocycle condition guarantees the existence of an extension to higher-order simplices. Thus 
\begin{equation}
\label{eq:cocycle_map}
\theta \colon \mathrm{VR}_\varepsilon(X) \longrightarrow \mathbb{R}/\mathbb{Z}\cong S^1
\end{equation}
is well defined.

\paragraph{Harmonic smoothing}
The resulting map \eqref{eq:cocycle_map} depends only on the representative of the real cohomology class
$[\tilde{\alpha}] \in H^1(\mathrm{VR}_\varepsilon(X);\mathbb{R})$ and not on the
chosen representative of the integral class $[\alpha]$. Indeed, modifying
$\alpha$ within its integral cohomology class changes $f$ by an
integer-valued $0$-cochain, which vanishes modulo $\mathbb{Z}$. However,
different choices of real representative $\tilde{\alpha}$ within the same real
cohomology class lead to different vertex functions $f$, and hence to different
coordinates.

To obtain a canonical coordinate, a distinguished representative is selected within the
affine space $\alpha + \operatorname{im}\delta^0$: the harmonic
cocycle, characterized as the unique element whose squared values
on edges are minimal. Hence, this representative is obtained by minimizing
\[
E(\tilde{\alpha}) := \sum_{e \in \mathrm{VR}_\varepsilon(X)^{(1)}}
\tilde{\alpha}(e)^2,
\]
with $\mathrm{VR}_\varepsilon(X)^{(1)}$ being the set of 1-simplices of $\mathrm{VR}_\varepsilon(X)$, which yields a unique solution $f$, up to addition of a constant, provided that
every vertex is incident to at least one edge. This minimizer can be computed
explicitly via a linear least-squares problem. The degree-one coordinate is then
defined from this canonical cocycle. Putting everything together yields the following result.
 
\begin{theorem}[Circular topological coordinates \cite{desilva2011circular}]
Let $(X,d)$ be a finite metric space and let $\varepsilon>0$.
Suppose that
\[
[\alpha]\in H^1(\mathrm{VR}_\varepsilon(X);\mathbb{Z})
\]
is an integral cohomology class arising from a persistent one-dimensional feature.
Then there exists a continuous map
\[
\theta:\mathrm{VR}_\varepsilon(X)\longrightarrow S^1
\]
with the following properties:
 
\begin{enumerate}
\item $\theta$ represents the class $[\alpha]$ under the Brown-representability correspondence between homotopy classes of maps $\mathrm{VR}_\varepsilon(X)\to S^1$ and $H^1(\mathrm{VR}_\varepsilon(X);\mathbb{Z})$.
 
\item $\theta$ is obtained from a real-valued cocycle representative of $[\alpha]$: the $0$-cochain $f\in C^0(\mathrm{VR}_\varepsilon(X);\mathbb{R})$ minimizing the harmonic energy $E$ determines the canonical representative $\widetilde{\alpha}=\alpha+\delta^0 f$, and $\theta$ is given on vertices by $x\mapsto f(x)\bmod\mathbb{Z}$.
 
\item The harmonic cocycle $\widetilde{\alpha}$ minimizing $E$ is unique; the potential $f$, and hence $\theta$, is thereby determined up to an additive constant on each connected component of $\mathrm{VR}_\varepsilon(X)$, that is, up to a rotation of $S^1$ on each component (a single global rotation when the complex is connected). In particular, any two maps produced by this procedure are homotopic.
\end{enumerate}
 
The restriction
\[
\theta|_X : X \longrightarrow S^1
\]
defines a well-defined circular coordinate on the data.
\end{theorem}

While this construction provides a principled realization of persistent
cohomological features as degree-one coordinates, the harmonic smoothing step is
sensitive to non-uniform sampling density. Regions of higher point density tend to
exhibit reduced coordinate variation, leading to systematic distortions. Further on in \cref{sec:density_robust_circular_coords}, we will review  workaround to this limitation via a density-robust subsampling and alignment framework \cite{blumberg2024subsampling}. 

\subsection{Spherical Topological Coordinates}
\label{sec:spherical_coords}

The first steps of the higher order, degree-two construction coincide with those of the degree-one case: for some input metric space $(X,d)$, compute persistent cohomology in degree two for the Vietoris--Rips filtration, select a prominent persistence interval, and fix some parameter value $\epsilon>0$ within its lifetime. This yields an integral class $[\alpha] \in H^2(\mathrm{VR}_\varepsilon(X);\mathbb{Z})$ by lifting the $\mathbb{F}_p$-cohomology class $[\alpha_p] \in H^2(\mathrm{VR}_\varepsilon(X);\mathbb{F}_p)$ associated to the chosen prominent interval.

\paragraph{A caveat}
However, unlike in the degree-one case and as mentioned previously, the target space $S^2$ is not an Eilenberg--MacLane
space of type $K(\mathbb{Z},2)$. By Brown representability, the class
\(
[\alpha] \in H^2(\mathrm{VR}_\varepsilon(X);\mathbb{Z})
\)
corresponds instead to a homotopy class of maps
\[
\mathrm{VR}_\varepsilon(X) \longrightarrow K(\mathbb{Z},2)
\simeq \mathbb{C}P^\infty.
\]
To obtain a sphere-valued representative, consider the canonical inclusion
$S^2 \hookrightarrow \mathbb{C}P^\infty$. Obstruction theory shows that the
associated lifting problem admits a solution after restricting to the
$3$-skeleton~\cite{schonsheck2024spherical}. 
Since
\[
H^2(\mathrm{VR}_\varepsilon(X);\mathbb{Z})
\cong
H^2(\mathrm{sk}_3(\mathrm{VR}_\varepsilon(X));\mathbb{Z}),
\]
no information is lost by this restriction.

\paragraph{Harmonic smoothing}
In contrast to the degree-one case, there is no intermediate step passing to a real-valued cocycle and integrating it to obtain vertex values.
Instead, the integer cocycle $\alpha \in Z^2(\mathrm{VR}_\varepsilon(X);\mathbb{Z})$
is used directly to construct an initial topological representative.

Concretely, a continuous map
\[
\widetilde{\theta}_0 :
\mathrm{sk}_3(\mathrm{VR}_\varepsilon(X))
\longrightarrow S^2
\]
is first defined as follows: all vertices and edges are sent to a fixed basepoint of $S^2$,
and for each oriented $2$-simplex $\sigma$, the boundary $\partial \sigma$
is collapsed to this basepoint and the resulting quotient $S^2 \cong \sigma/\partial\sigma$
is mapped to $S^2$ by the canonical degree $\alpha(\sigma)$ map.
The cocycle condition guarantees compatibility on $3$-simplices, so that
$\widetilde{\theta}_0$ is well defined and represents $[\alpha]$.
However, its restriction to the vertex set is constant and therefore does not
yet provide a meaningful coordinate on the data.

Unlike the degree-one setting, where smooth representatives are obtained by solving a linear least-squares problem for harmonic cocycles, no analogous linear construction exists for sphere-valued maps. To obtain a smooth sphere-valued representative, the map itself is optimized
within its homotopy class. This approach retains the variational viewpoint underlying the
degree-one construction, where a cohomologous representative was found by
minimizing a quadratic energy on edges, whereas here, an energy measuring the variation of a sphere-valued map across
triangles is minimized. For a map
\[
\theta:\mathrm{sk}_2(\mathrm{VR}_\varepsilon(X))\longrightarrow S^2,
\]
the smoothing objective is the \emph{harmonic energy}
\[
E_H(\theta)
\;:=\;
\sum_{\sigma\in \mathrm{VR}_\varepsilon(X)^{(2)}}
\frac12\,\bigl|A(\theta(\sigma))\bigr|^2,
\]
where $\mathrm{VR}_\varepsilon(X)^{(2)}$ denotes the set of 2-simplices and $A(\theta(\sigma))$ the signed area of the spherical triangle
spanned by the image of $\sigma$. The harmonic energy plays the role of a Dirichlet energy and provides a variational criterion for selecting a smooth representative within the prescribed homotopy class of sphere-valued maps.

Energy minimizers need not be unique; they are at least ambiguous under postcomposition by conformal automorphisms. To obtain a canonical representative of the resulting equivalence class, a gauge-fixing condition is imposed via \emph{Möbius centering}: the representative is chosen so that the (spherical) center of mass of the vertex images vanishes. Starting from $\theta_0:=\widetilde{\theta}_0|_{\mathrm{sk}_2}$, an iterative (projected) gradient-descent scheme is carried out on
$\mathrm{sk}_2(\mathrm{VR}_\varepsilon(X))$ which decreases $E_H$ while enforcing
M\"obius centering at each step.

Since the harmonic energy depends only on the images of 2-simplices, the optimization is carried out on the 2-skeleton. However, the resulting map must still represent the prescribed degree-two homotopy class, which requires it to extend over the 3-skeleton. After convergence, the resulting map on
$\mathrm{sk}_2$ is verified (and enforced) to admit an extension
\[
\theta:\mathrm{sk}_3(\mathrm{VR}_\varepsilon(X))\longrightarrow S^2
\]
in the same homotopy class as $\widetilde{\theta}_0$. Restricting $\theta$ to
the vertex set $X$ then yields the desired sphere-valued coordinate on the
data.

\begin{theorem}[Spherical topological coordinates \cite{schonsheck2024spherical}]
Let $(X,d)$ be a finite metric space and let $\varepsilon>0$. 
Suppose that
\[
[\alpha]\in H^2(\mathrm{VR}_\varepsilon(X);\mathbb{Z})
\]
is an integral cohomology class arising from a persistent two-dimensional feature.
Then there exists a continuous map
\[
\theta:\mathrm{sk}_3(\mathrm{VR}_\varepsilon(X))\longrightarrow S^2
\]
with the following properties:

\begin{enumerate}
\item $\theta$ represents the class $[\alpha]$ under the correspondence induced by Brown representability and the lift through $S^2\hookrightarrow \mathbb{C}P^\infty$.

\item $\theta$ is obtained from the canonical topological representative 
$\widetilde{\theta}_0$ by minimizing the harmonic energy $E_H$ within its homotopy class, subject to M\"obius centering.

\item Any two maps produced by this procedure are homotopic and differ at most by postcomposition with a conformal automorphism of $S^2$.
\end{enumerate}

In particular, the restriction
\[
\theta|_X : X \longrightarrow S^2
\]
defines a well-defined sphere-valued coordinate on the data.
\end{theorem}

As in the degree-one case, this construction is sensitive to non-uniform sampling density. Regions with higher point density tend to exert a higher influence on the energy minimization, leading to distortions in the resulting spherical coordinate. The following section addresses existing approaches this limitation.

\subsection{Density-Robust Circular Coordinates}
\label{sec:density_robust_circular_coords}

One successful approach to mitigating the sensitivity of classical circular
coordinates to non-uniform sampling is based on a subsampling strategy
introduced by Blumberg et al.~\cite{blumberg2024subsampling}. Rather than
computing a single coordinate on the full data set, their method constructs
circular coordinates on multiple carefully chosen subsamples and aggregates
the resulting maps.


\paragraph{Subsampling}
Let $(X,d)$ be a finite metric space sampled from an underlying space with
density function $\rho$. The aim is to generate subsamples $X_1,\dots,X_M \subset X$  that are uniformly distributed via a rejection
sampling scheme. To this end, an
empirical density estimator $\widehat{\rho}:X\to\mathbb{R}_{>0}$ is computed,
for instance using local neighborhood statistics. Under suitable assumptions,
$\widehat{\rho}$ can consistently estimate the true sampling
density $\rho$. Points are then retained with probability proportional to
$\widehat{\rho}(x)^{-1}$, so that densely sampled regions are thinned out
while sparsely sampled regions are preserved. As a consequence, the resulting
subsamples approximate uniform samples from the underlying space.

On each $X_i$, the circular topological coordinate construction from
\cref{sec:circular_coords} is applied: persistent cohomology
in degree one is computed, a prominent class
\[
[\alpha_i] \in H^1(\mathrm{VR}_\varepsilon(X_i);\mathbb{Z}),
\]
is selected and finally, a circle-valued map
\[
\theta_i : X_i \longrightarrow S^1
\]
is constructed via lifting, harmonic smoothing, and integration.

\paragraph{Extending}
To combine the information obtained from different subsamples, the maps
$\theta_i$ must be defined on the same domain. Since each coordinate is
initially defined only on the subsample $X_i$, it is  extended to the entire data
set $X$ using a nearest-neighbor procedure with Gaussian kernel weights. For
$x\in X$, define
\[
\theta_i(x)
:=
\frac{\sum_{y\in X_i} K_\sigma(d(x,y))\,\theta_i(y)}
{\sum_{y\in X_i} K_\sigma(d(x,y))},
\qquad
K_\sigma(r)=\exp\!\left(-\frac{r^2}{2\sigma^2}\right),
\]
so that nearby points in the subsample contribute most significantly to the
extension. This produces maps $\theta_i:X\to S^1$ defined on the entire data
set while preserving the local geometric structure encoded by the subsample
coordinates.

\paragraph{Aligning}
Since circular coordinates are defined only up to rotation of $S^1$, the
family of maps $\theta_1,\dots,\theta_M:X\to S^1$ must be aligned before
aggregation. Moreover, the aggregation itself should determine a consensus
configuration that best represents all maps simultaneously. This leads to a
joint alignment and averaging problem belonging to a class of optimization
problems known as \emph{Procrustes problems} \cite{gower2004procrustes}. Procrustes problems
seek rigid transformations that align a collection of configurations while
simultaneously determining the configuration that is closest to all of them in the sense of least squares.

In the classical orthogonal Procrustes problem in $\mathbb{R}^2$, given
vectors $v_1,\dots,v_M\in\mathbb{R}^2$ the task is to find orthogonal transformations
$R_1,\dots,R_M\in SO(2)$ together with a consensus vector $\bar v$ that
minimize the squared Euclidean discrepancy
\[
\sum_{i=1}^M \|R_i v_i-\bar v\|^2 .
\]
This formulation aligns the configurations while simultaneously estimating
their best average configuration.

In the setting of degree-one coordinates, the values of the maps $\theta_i$ lie on the circle
$S^1\subset\mathbb{R}^2$. Writing each value as a unit vector in
$\mathbb{R}^2$, the alignment problem can be formulated as follows.

\begin{definition}[Circular Procrustes problem]
Let $\theta_1,\dots,\theta_M : X \to S^1$ be circle-valued maps. The \emph{circular
Procrustes problem} consists of finding rotations
$R_1,\dots,R_M\in SO(2)$ and a consensus configuration
$\bar\theta:X\to S^1$ that minimize
\[
\sum_{i=1}^M \sum_{x\in X}
d_{S^1}\!\big(R_i\theta_i(x),\,\bar\theta(x)\big)^2,
\]
where $d_{S^1}$ denotes the geodesic distance on the circle.
\end{definition}

Directly solving the circular Procrustes problem with respect to the
geodesic metric is computationally challenging because the geodesic distance gives rise to a nonlinear, nonconvex optimization problem, in contrast to the Euclidean formulation, which reduces to a classical least-squares problem admitting an SVD-based solution. It is possible to bypass these challenges
in two stages. First, the corresponding orthogonal Procrustes problem in
$\mathbb{R}^2$ is solved using the Euclidean metric. This problem admits a
global solution that can be computed efficiently using standard linear
algebra methods. The resulting rotations provide an initial alignment of
the maps. In a second step, a local hill-climbing procedure is applied to
refine the rotations with respect to the geodesic distance on $S^1$. Blumberg et al.~\cite{blumberg2024subsampling} prove that the solution of the
orthogonal Procrustes problem in $\mathbb{R}^2$ is close to the optimal
solution of the circular problem. This result justifies the two-stage
procedure and ensures that the hill-climbing step only needs to perform a
small local correction.

After alignment, the maps are averaged to obtain a final density-robust degree-one coordinate
\[
\theta : X \longrightarrow S^1 .
\]

\section{Density-Robust Spherical Topological Coordinates}
\label{sec:density_robust_spherical}

This section develops the first density-robust construction of spherical topological coordinates. Extending existing density-robust methods from degree-one to degree-two coordinates requires overcoming three fundamental obstacles. First, unlike the circular setting, the weighting strategy of Paik and Park \cite{paik2023circular} cannot be transferred because spherical coordinates are obtained by directly optimizing sphere-valued maps rather than harmonic cocycles. Second, sphere-valued coordinates computed on different subsamples must be aligned on $S^2$, which requires a new formulation of the spherical Procrustes problem. Third, because this optimization is intrinsically nonlinear, new approximation guarantees showing that it is well approximated by an efficiently computable Euclidean Procrustes problem are required. 

Our resulting construction overcomes these challenges to provide a practical and theoretically justified framework for constructing density-robust spherical topological coordinates: we combine density-balanced subsampling with a new spherical alignment and averaging procedure. 
The principal theoretical contribution is the formulation and analysis of a spherical Procrustes problem for jointly aligning sphere-valued coordinates, together with an approximation theorem showing that its solution is well approximated by an efficiently computable Euclidean relaxation.



\subsection{Subsampling and Coordinate Extension}


A natural first attempt at constructing density-robust spherical coordinates would be to avoid subsampling altogether and instead extend the weighting-based circular method of Paik and Park~\cite{paik2023circular} to the spherical setting. In the circular construction, density robustness is achieved by introducing density weights into the harmonic energy whose minimizer selects the real cocycle representative defining the coordinate. Specifically, the edge energy
$$
\sum_e \tilde \alpha(e)^2
$$
is replaced by a weighted energy
$$
\sum_e w_e \tilde \alpha(e)^2
$$
yielding a weighted harmonic cocycle and hence a density-corrected circular coordinate.

This construction, however, relies fundamentally on the existence of a harmonic cocycle representation. The weighting approach therefore does not naturally extend to the spherical construction because the algebraic object to which the weights are applied in the circular setting---the harmonic cocycle---has no analogue in the sphere-valued variational formulation. As discussed in \cref{sec:spherical_coords} the spherical construction never passes to a real cocycle, but instead directly optimizes a sphere-valued map under the area-based harmonic energy. Consequently, there is no natural location at which density weights can be introduced.

We therefore adopt a subsampling-based strategy which is compatible with the variational formulation of spherical coordinates and leaves the underlying sphere-valued optimization unchanged. Besides avoiding this obstruction, the subsampling approach has the additional advantage that it leaves the underlying spherical coordinate construction unchanged. We begin by generating approximately uniform subsamples of the data via rejection sampling. Since this step forms the foundation of the density-robust construction, we briefly recall the procedure for completeness.

\paragraph{Rejection sampling} 
Suppose our data set, the finite metric space $(X,d)$ with $X \subset \mathbb{R}^N$, is sampled from a probability
density $\rho$ supported on a $m$-dimensional submanifold
$\Mcal \subset \mathbb{R}^N$, equipped with its Riemannian volume measure $\mu$.
If points are accepted with probability inversely proportional to the
sampling density, then densely sampled regions are thinned out while
sparsely sampled regions are preserved.
 
More precisely, let $\pi:X\to[0,1]$ be an acceptance function.
A subsample $Y\subset X$ is obtained by accepting each point
$x\in X$ independently with probability $\pi(x)$.
If $\pi(x) \propto 1/\rho(x)$, then for any Borel set $U\subset \Mcal$ the
distribution of accepted points satisfies
\[
\mathbb{P}(x\in U \mid x \text{ accepted})
=
\frac{\int_U \pi(x)\rho(x)\,d\mu(x)}
{\int_\Mcal \pi(x)\rho(x)\,d\mu(x)}
\propto
\frac{\mu(U)}{\mu(\Mcal)} .
\]
Thus the accepted points are distributed uniformly on $\Mcal$.
Rejection sampling therefore provides a way to approximately remove the
effect of non-uniform sampling density.
 
In practice, the density $\rho$ is unknown and must be estimated from the
sample $X$. Blumberg et al.~\cite{blumberg2024subsampling} use the following simple local
counting estimator: For $\varepsilon>0$, define
\[
\widehat{\rho}_\varepsilon(x)
=
\frac{\#\bigl(X \cap B_\varepsilon(x)\bigr)}
{n\,V_m\,\varepsilon^m},
\]
where $B_\varepsilon(x)$ denotes the Euclidean ball of radius $\varepsilon$
centred at $x$, $n=\#X$ is the sample size, and $V_m$ is the volume of the
$m$-dimensional unit ball. This estimator counts the number of sample points in a neighborhood of
$x$ and rescales by the expected volume of that neighborhood.
Regions where many points lie close together therefore receive a larger
estimated density.
 
\begin{theorem}[Consistency of the density estimator {\cite[Prop.~3.3]{blumberg2024subsampling}}]
Let $\Mcal\subset\mathbb{R}^N$ be a compact $m$-dimensional $C^2$ submanifold
with volume measure $\mu$, and let $\rho$ be a probability density on $\Mcal$.
Let $X=\{x_1,\dots,x_n\}$ be i.i.d.\ samples from $\rho$, and define
\[
\widehat{\rho}_\varepsilon(x)
=
\frac{\#(X\cap B_\varepsilon(x))}
{n V_m \varepsilon^m}.
\]
If $\varepsilon_n \to 0$ and
$\varepsilon_n n^{1/(2m)} \to \infty$ as $n\to\infty$, then for every
continuity point $x$ of $\rho$,
\[
\widehat{\rho}_{\varepsilon_n}(x)
\;\xrightarrow{p}\;
\rho(x).
\]
\end{theorem}
 
The proof is given in \cite{blumberg2024subsampling}.
For the purpose of rejection sampling, the density estimator is only
required up to a multiplicative constant, since the acceptance probability
must be proportional to $1/\rho(x)$. The normalization factor can therefore
be omitted which results in the simplified estimator
\[
\widehat{\rho}_\varepsilon(x)
=
\#\bigl(X \cap B_\varepsilon(x)\bigr).
\]
The acceptance probability is then given by
\[
\pi(x)
=
\frac{C}{\widehat{\rho}_\varepsilon(x)},
\]
where $C \le \min_{x\in X} \widehat{\rho}_\varepsilon(x)$ is a scaling
parameter controlling the expected size of the subsample.
Applying this procedure independently to all points in $X$ yields a
subsample whose distribution approximates the uniform distribution on $\Mcal$.
Repeating the process produces multiple approximately uniform subsamples
$X_1,\dots,X_M$ that form the basis for the subsampling-based coordinate
construction.
 
\paragraph{Extending subsample coordinates} Having constructed approximately uniform subsamples
$X_1,\dots,X_M \subset X$, the next step is to compute spherical coordinates
on each subsample as discussed in \cref{sec:spherical_coords}. In particular, each
$i \in \{1,\dots,M\}$ yields a map
\[
\Phi_i : X_i \longrightarrow S^2
\]
derived from a prominent persistent cohomology class in degree two.
 
As in the degree-one framework, these coordinates are initially defined only
on the subsamples themselves. In order to combine the information obtained
from different subsamples, the maps must first be extended to the entire
dataset. We use the nearest-neighbor procedure with Gaussian kernel
weights as in the degree-one case: For $x\in X$, we set
\[
\Phi_i(x)
=
\frac{\sum_{y\in X_i} K_\sigma(d(x,y))\,\Phi_i(y)}
{\bigl\|\sum_{y\in X_i} K_\sigma(d(x,y))\,\Phi_i(y)\bigr\|},
\qquad
K_\sigma(r)=\exp\!\left(-\frac{r^2}{2\sigma^2}\right),
\]
so that nearby points in the subsample contribute most strongly to the
extension. Here, the kernel-weighted average of the unit vectors
$\phi_i(y)\in S^2\subset\mathbb{R}^3$ is renormalized to the sphere, ensuring
that the result again lies on $S^2$. This produces maps
\[
\Phi_i : X \longrightarrow S^2
\]
defined on the entire data set while preserving the local geometric structure
encoded by the subsample coordinates.
 
\subsection{Aligning: Solving the Spherical Procrustes Problem}
The remaining task is to align the extended coordinate functions
$\Phi_1,\dots,\Phi_M$ and to combine them into a single representative
coordinate. This is achieved by solving a spherical analogue of the
generalized Procrustes problem.
 
\begin{definition}[Spherical Procrustes problem]
Let $\Phi_1,\dots,\Phi_M : \{1,\dots,n\} \to S^2$ be configurations of
$n$ points on the sphere. The \emph{spherical Procrustes problem} seeks rotations
$R_1,\dots,R_M \in SO(3)$ together with a centroid configuration
$\Theta : \{1,\dots,n\} \to S^2$ that minimize the loss
\[
L(R_1,\dots,R_M,\Theta)
=
\frac{1}{M}
\sum_{i=1}^M
\sum_{j=1}^n
d_{S^2}\!\big(R_i \Phi_i(j), \Theta(j)\big)^2 ,
\]
where $d_{S^2}$ denotes the geodesic distance on the sphere.
\end{definition}
 
This formulation generalizes the circular Procrustes problem introduced previously by replacing rotations of the circle with rotations of
the sphere. Intuitively, the optimization seeks rotations that best align
the coordinate functions while simultaneously determining a centroid
configuration that represents their average.
 
As in the circular setting, solving the problem directly with respect to the
geodesic metric on the sphere leads to a nonconvex optimization problem for which no closed-form solution is available. Instead, we first solve the corresponding orthogonal Procrustes problem in $\mathbb{R}^3$, viewing each configuration as a set of unit vectors embedded in Euclidean space. This relaxation admits a closed-form solution via the singular value decomposition and provides an initial alignment, which is subsequently refined by local optimization of the original spherical loss. A similar Euclidean initialization strategy was employed in \cite{blumberg2024subsampling}.
 

The Euclidean relaxation remains a provably accurate surrogate for the spherical alignment problem. Specifically, the Euclidean objective provides a controlled approximation to the geodesic objective, which provides a theoretically justified initialization for the subsequent local optimization.  The key observation is that every geodesic on $S^2$ is contained in a great circle, reducing the comparison between geodesic and Euclidean distances to the one-dimensional setting considered previously. 

\begin{lemma}
\label{lem:spherical-vs-euclidean-distance}
Let \( x, y \in S^{2} \subset \mathbb{R}^{3} \), and let \( r \in [0,1] \). Then
\[
  d_{S^{2}}(x, y) \;\le\; \pi \cdot \| x - r y \|,
\]
where \( d_{S^{2}} \) denotes the geodesic distance on the sphere and \( \| \cdot \| \) the Euclidean norm in \( \mathbb{R}^{3} \).
\end{lemma}
 
\begin{proof}


Every pair of points on $S^2$ is contained in a great circle, and every great circle is canonically isometric to $S^1$.  Let $H \subset \mathbb{R}^{3}$ denote the plane through the origin containing $x$ and $y$.  Then $H \cap S^2$ is precisely this great circle.  Since the geodesic joining $x$ and $y$ is contained entirely $H \cap S^2$,
$
d_{S^2}(x,y) = d_{S^1}(x,y).
$
Applying Lemma 4.3 of \cite{blumberg2024subsampling} to this copy of
$S^1$ yields
$
d_{S^1}(x,y)
\le
\pi \|x-ry\|.
$
Since
$
d_{S^2}(x,y)=d_{S^1}(x,y),
$
the desired estimate follows.
\end{proof}


 

The previous lemma provides a pointwise comparison between the spherical geodesic distance and its Euclidean counterpart.  By summing these estimates over all landmarks and configurations, we obtain a corresponding comparison between the spherical and Euclidean Procrustes objectives.  This yields the following approximation result, which is the principal theoretical justification for our alignment procedure.  It shows that the solution of the Euclidean Procrustes problem provides a controlled approximation to the optimal spherical alignment, with an error bounded by the optimal spherical loss itself.

\begin{theorem}[Approximation of the spherical Procrustes solution by Euclidean alignment]
\label{prop:procrustes-approximation-sphere}
Let \( \Phi_{1}, \dots, \Phi_{M} \) be \( M \) configurations of \( n \) points on the sphere \( S^{2} \).  
Let \( ((g_{1}, \dots, g_{M}), \Theta) \) denote the solution to the generalized Procrustes problem on the sphere, and let 
\( ((g'_{1}, \dots, g'_{M}), \Psi) \) denote the solution to the \( O(3) \) Procrustes problem in 
\( \mathbb{R}^{3} \) for the same configurations.  
Let \( \Theta' \) be the projection of \( \Psi \) back to \( S^{2} \).  
Then
\[
  \inf_{h \in O(3)} d_{(S^{2})^{n}}(h \cdot \Theta, \Theta') 
  \;\le\; (1 + \pi) \, \sqrt{L^{*}},
\]
where \( L^{*} = L((g_{1}, \dots, g_{M}), \Theta) \) denotes the minimal value of the spherical Procrustes loss functional.
\end{theorem}
 
\begin{proof}
By the triangle inequality and \cref{lem:spherical-vs-euclidean-distance},
\[
\begin{aligned}
  \inf_{h \in O(3)} d_{(S^{2})^{n}}(h \cdot \Theta, \Theta') 
  &= \inf_{h, h' \in O(3)} d_{(S^{2})^{n}}(h \cdot \Theta, h' \cdot \Theta') \\
  &\le \inf_{h, h'} \bigl\{
      d_{(S^{2})^{n}}(h \cdot \Theta, \Phi_{i}) +
      d_{(S^{2})^{n}}(h' \cdot \Theta', \Phi_{i})
    \bigr\} \\
  &= \inf_{h} d_{(S^{2})^{n}}(h \cdot \Phi_{i}, \Theta) +
     \inf_{h} d_{(S^{2})^{n}}(h \cdot \Phi_{i}, \Theta') \\
  &\le \inf_{h} d_{(S^{2})^{n}}(h \cdot \Phi_{i}, \Theta) +
        \inf_{h} \pi \cdot d_{(\mathbb{R}^{3})^{n}}(h \cdot \Phi_{i}, \Psi).
\end{aligned}
\]
Applying the triangle inequality in \( \mathbb{R}^M \) yields

\begin{align*}
  \inf_{h \in O(3)} d_{(S^2)^n}(h \cdot \Theta, \Theta') 
  &= \left(\frac{1}{M}\sum_{i=1}^{M}\left[\inf_{h \in O(3)} d_{(S^2)^n}(h \cdot \Theta, \Theta')\right]^2\right)^{1/2}\\
  &\leq \left( \frac{1}{M} \sum_{i=1}^{M}\left[\inf_{h} d_{(S^2)^n}(h \cdot \Phi_i, \Theta) + \inf_{h} \pi\cdot d_{(\mathbb{R}^3)^n}(h \cdot \Phi_i, \Psi) \right]^2 \right)^{1/2}\\
  &\le 
  \left( \frac{1}{M} \sum_{i=1}^{M} \left[\inf_{h} d_{(S^2)^n}(h \cdot \Phi_i, \Theta)\right]^2 \right)^{1/2} \\
  & \qquad \qquad \qquad \qquad \qquad 
  + \pi \left( \frac{1}{M} \sum_{i=1}^{M} \left[\inf_{h} d_{(\mathbb{R}^3)^n}(h \cdot \Phi_i, \Psi)\right]^2 \right)^{1/2}.
\end{align*}

The first term equals \( \sqrt{L^*} \) by definition.  
For the second term, let \( L_{\mathbb{R}^3} \) denote the Euclidean Procrustes loss:
\[
  L_{\mathbb{R}^{3}}((h_1, \dots, h_M), \Xi)
  = \frac{1}{M} \sum_{i=1}^{M} d_{(\mathbb{R}^{3})^n}(h_i \cdot \Phi_i, \Xi)^2.
\]
By optimality of \( ((g'_1, \dots, g'_M), \Psi) \),
\[
  \frac{1}{M} \sum_{i=1}^{M} \inf_{h} d_{(\mathbb{R}^{3})^n}(h \cdot \Phi_i, \Psi)^2
  = L_{\mathbb{R}^{3}}((g'_1, \dots, g'_M), \Psi)
  \le L_{\mathbb{R}^{3}}((g_1, \dots, g_M), \Theta).
\]
Since \( d_{\mathbb{R}^3} \le d_{S^2} \) on \( S^2 \),
\[
  L_{\mathbb{R}^{3}}((g_1, \dots, g_M), \Theta)
  \le L((g_1, \dots, g_M), \Theta) = L^*.
\]
Combining these bounds gives
\[
  \inf_{h \in O(3)} d_{(S^2)^n}(h \cdot \Theta, \Theta') 
  \le \sqrt{L^*} + \pi \sqrt{L^*}
  = (1 + \pi) \sqrt{L^*},
\]
as required.
\end{proof}

\section{Experiments}
\label{sec:experiments}
This section presents an experimental evaluation of the proposed method.  After describing the implementation and experimental setup, we evaluate its computational efficiency and the robustness of the recovered spherical coordinates under both uniform and non-uniform sampling.

\subsection{Experimental Setup}
\paragraph{Terminology} Throughout the experiments we distinguish three constructions. We refer to the original spherical coordinate method of Schonsheck and Schonsheck~\cite{schonsheck2024spherical} as the \emph{baseline}; to the same method equipped with the two implementation accelerations described below as the \emph{optimized baseline}; and to the density-robust construction of \cref{sec:density_robust_spherical}, which augments the optimized baseline with rejection sampling and spherical alignment, as the \emph{subsampling method}. Since the two accelerations affect only the implementation, the baseline and optimized baseline compute identical coordinates and differ only in runtime.

\paragraph{Implementation} Our implementation builds on the spherical coordinate construction of Schonsheck and Schonsheck~\cite{schonsheck2024spherical} and the subsampling framework of Blumberg et al.~\cite{blumberg2024subsampling}. Persistent cohomology is computed using \texttt{Ripser}~\cite{bauer2021ripser}, while spherical coordinates are constructed via the variational optimization procedure described in \cref{sec:spherical_coords}. Since this optimization is performed repeatedly across many subsamples, computational efficiency is essential for the scalability of the method. We therefore introduced two implementation optimizations. First, the construction of $2$-simplices in the Vietoris--Rips complex is accelerated by replacing brute-force triangle enumeration with a sparse adjacency-based algorithm using intersections of neighborhood sets. Second, the variational optimization routine is accelerated through just-in-time compilation using \texttt{numba}~\cite{lam2015numba}. These modifications affect only computational efficiency and leave the underlying coordinate construction unchanged. 

The alignment stage follows the two-stage procedure developed in \cref{sec:density_robust_spherical}. An initial alignment is obtained by solving the orthogonal Procrustes problem in $\mathbb{R}^3$ using the method of ten Berge~\cite{tenberge1986}, providing a global initialization for the rotational degrees of freedom. This initialization is subsequently refined by solving the spherical Procrustes problem using Fréchet means~\cite{petersen2016riemannian,miolane2020geomstats} together with constrained optimization on $SO(3)$~\cite{virtanen2020scipy}. This combines the computational efficiency of Euclidean alignment with the intrinsic geometry of the sphere and provides stable convergence throughout the experiments.

\paragraph{Experimental protocol} All experiments were run on a single desktop machine with an AMD Ryzen~5 7600X
CPU and 32\,GB of RAM under Debian GNU/Linux~12; although a GPU was present, none
of the reported computations used GPU acceleration. Runtimes were measured with
Python's \texttt{time} module. To support reproducibility, every point cloud was
generated from a fixed random seed, and each runtime measurement is reported as
the mean over five independent trials, with shaded regions in the corresponding
figures indicating one standard deviation.
 
The synthetic data fall into two groups, matching the two evaluation goals. For the runtime experiments of \cref{sec:runtime}, point clouds are sampled uniformly from the unit sphere and perturbed with Gaussian noise. For the coordinate-quality experiments of \cref{sec:quality}, we additionally generate non-uniform samples whose density is deliberately biased towards specific regions of the sphere using von Mises--Fisher distributions and related constructions; the exact distributions and their parameters are stated alongside each experiment. Throughout, recovered coordinates are evaluated by correlating the computed azimuth and elevation angles against the ground-truth angles of the underlying sphere, generalizing the single-angle correlation used to assess circular coordinates.
 
\subsection{Runtime and Scalability}
\label{sec:runtime}
 
We evaluate the computational performance of the proposed method in two stages. We first quantify the effect of the implementation optimizations on the baseline construction before assessing the runtime and scalability of the subsampling method.
 
\paragraph{Implementation optimizations}  The two accelerations leave the computed coordinates unchanged and affect only
runtime. \cref{fig:runtime-components} isolates their individual effect.
Just-in-time compilation accelerates the variational optimization stage, from
roughly $2300$\,s to $20$\,s at $200$ points, while replacing brute-force
triangle enumeration with the sparse, adjacency-based construction reduces the
cost of building $2$-simplices from cubic to roughly quadratic scaling; the
latter dominates at larger sizes, falling from about $30{,}000$\,s to one second
at $500$ points. Together (\cref{fig:runtime-full}) they reduce the full
pipeline at $500$ points from over ten and a half hours to under four minutes,
a gain of more than two orders of magnitude.
 
These speed-ups do not, however, extend the range of the baseline. Computing
degree-two cohomology requires storing all simplices up to dimension three,
whose number grows as $\mathcal{O}(n^4)$; on the $32$\,GB machine used here the
computation exhausts memory at roughly $850$ points, and extrapolation suggests
that even $1$\,TB would reach only about $3000$ points. This limitation is intrinsic to the construction rather than its implementation: faster code cannot remove it, and beyond a few hundred points the baseline becomes infeasible regardless of optimization.

\begin{figure}[htbp]
    \centering
    \begin{subfigure}{0.48\textwidth}
        \centering
        \includegraphics[width=\textwidth]{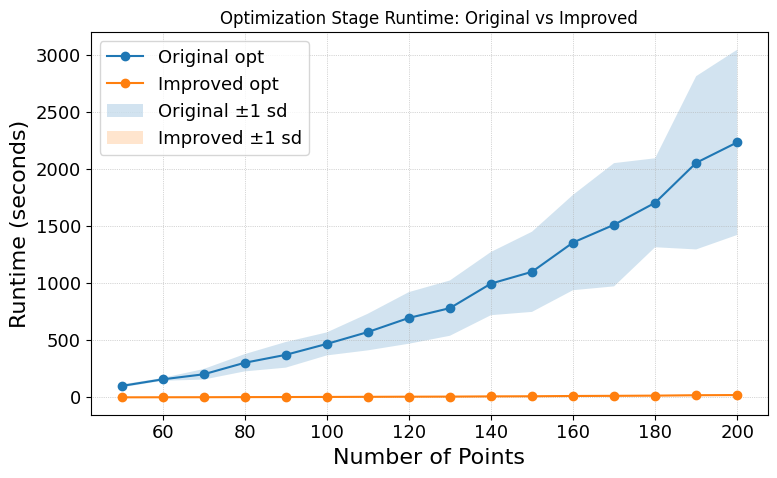}
        \caption{Optimization stage.}
        \label{fig:runtime-optimization}
    \end{subfigure}
    \hfill
    \begin{subfigure}{0.48\textwidth}
        \centering
        \includegraphics[width=\textwidth]{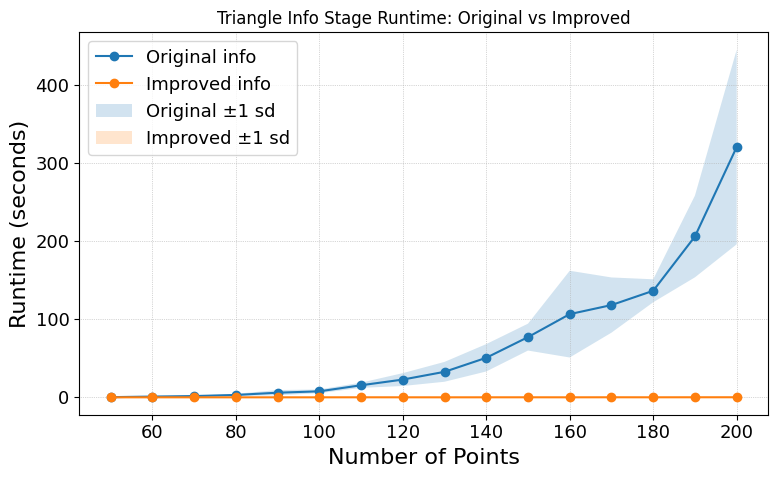}
        \caption{Triangle construction.}
        \label{fig:runtime-triangles}
    \end{subfigure}
    \caption{Effect of the two accelerations in isolation, comparing the
    baseline with the optimized baseline.}
    \label{fig:runtime-components}
\end{figure}
 
\begin{figure}[htbp]
    \centering
    \begin{subfigure}{0.48\textwidth}
        \centering
        \includegraphics[width=\textwidth]{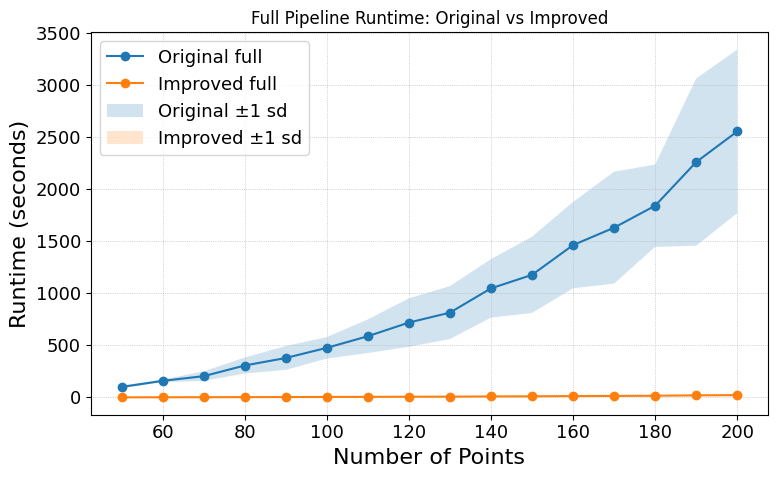}
        \caption{Total runtime.}
        \label{fig:runtime-total}
    \end{subfigure}
    \hfill
    \begin{subfigure}{0.48\textwidth}
        \centering
        \includegraphics[width=\textwidth]{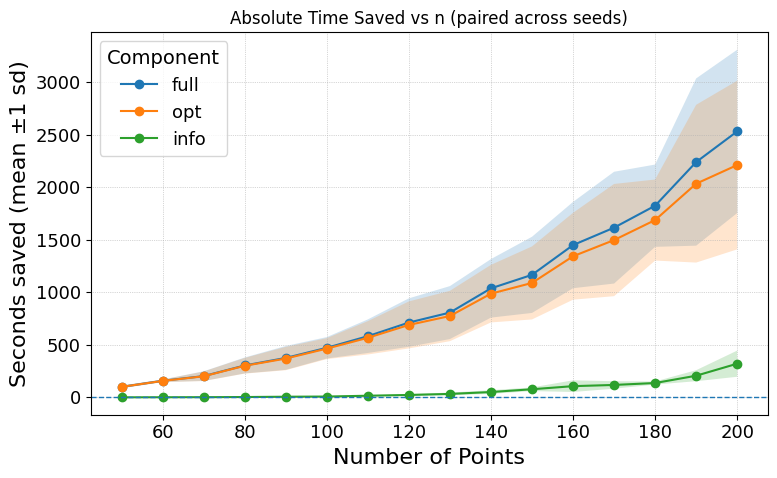}
        \caption{Absolute time saved.}
        \label{fig:runtime-saved}
    \end{subfigure}
    \caption{Overall runtime of the baseline and the optimized baseline on the
    full spherical coordinate pipeline.}
    \label{fig:runtime-full}
\end{figure}
 
\paragraph{Scalability of the subsampling method}
The subsampling method removes this limitation. \cref{fig:improved-vs-new}
compares the two directly on inputs of $100$ to $800$ points: the optimized
baseline is faster on small clouds, but its runtime climbs with the point count
while the subsampling method stays nearly flat, so the two cross at roughly
$400$ points and the gap then widens in the latter's favor. Because spherical coordinates are computed on fixed-size subsamples, runtime depends primarily on the number and size of the subsamples rather than on the total number of input points. Consequently, across datasets ranging from 500 to 10,000 points, runtime remains nearly constant (\cref{fig:pipeline-scaling}). Peak memory usage exhibits the same behavior: each subsample is processed independently before being discarded, so memory depends only on the subsample size rather than on the size of the full dataset. The per-subsample cost is still that of the underlying spherical coordinate construction. Consequently, the implementation optimizations described above remain essential, since they are applied once for every subsample.

Subsampling likewise resolves the memory bottleneck. Each subsample is processed
independently and its complex discarded before the next, so peak memory depends
only on the subsample size and stays fixed as the dataset grows. The
$\mathcal{O}(n^4)$ ceiling therefore disappears, and the method applies to
arbitrarily large inputs within a fixed memory budget. Simply lowering the
Vietoris--Rips threshold would also bound the complex, but at the cost of
discarding the topological features the coordinates depend on; subsampling
avoids this trade-off.
 
\begin{figure}[htbp]
    \centering
    \begin{subfigure}[b]{0.48\textwidth}
        \includegraphics[width=\textwidth]{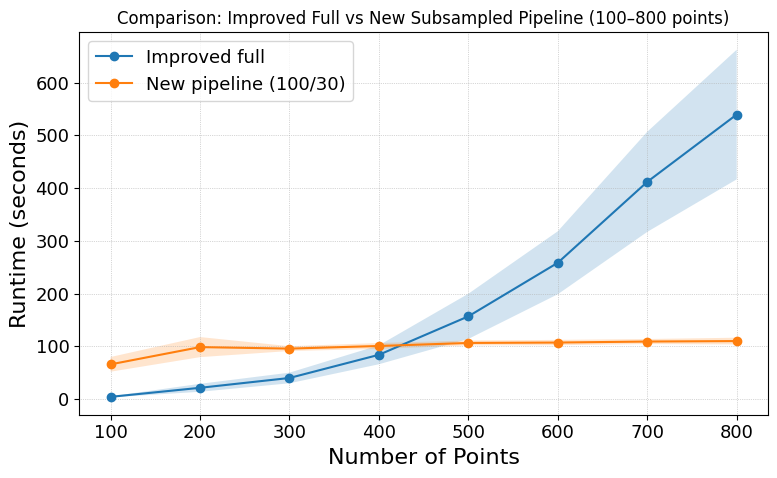}
        \caption{Optimized baseline vs.\ subsampling method ($100$ to $800$
        points; subsampling: $30$ subsamples of size $100$). The two cross at
        roughly $400$ points.}
        \label{fig:improved-vs-new}
    \end{subfigure}
    \hfill
    \begin{subfigure}[b]{0.48\textwidth}
        \includegraphics[width=\textwidth]{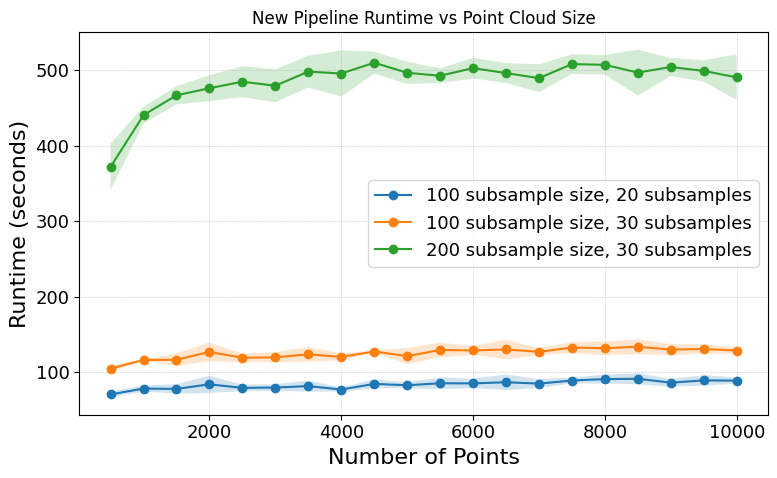}
        \caption{Subsampling method on large datasets ($500$ to $10{,}000$
        points) under several configurations; the cost is effectively
        independent of the total number of points.}
        \label{fig:pipeline-scaling}
    \end{subfigure}
    \caption{Runtime scaling of the subsampling method: a direct comparison
    against the optimised baseline on moderate inputs, and its near-constant
    runtime on large inputs.}
    \label{fig:runtime-scaling}
\end{figure}
 
\begin{figure}[htbp]
    \centering
    \includegraphics[width=0.5\textwidth]{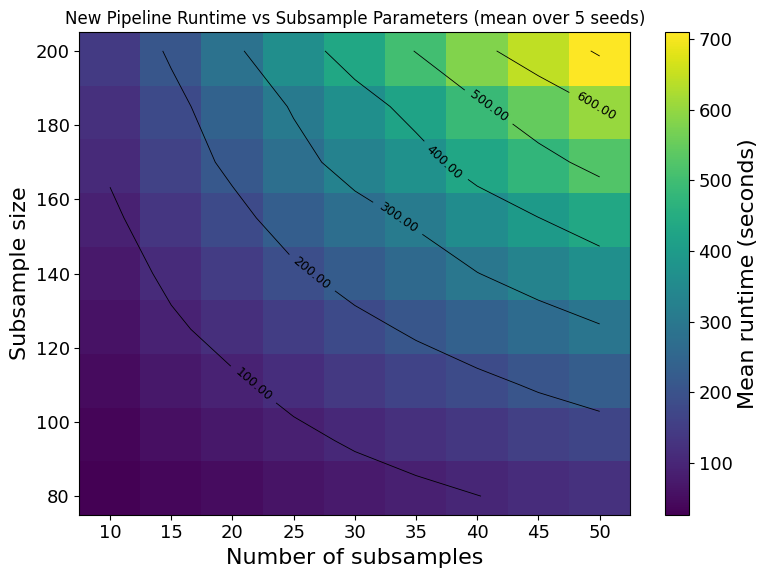}
    \caption{Runtime of the subsampling method as a function of subsample size
    and number of subsamples.}
    \label{fig:runtime-heatmap}
\end{figure}

\subsection{Coordinate Quality and Robustness}
\label{sec:quality}
 
We now turn from runtime to the quality of the recovered coordinates. Our main goal is to prove the robustness of the construction under non-uniform sampling where the baseline degrades, while our subsampling method still recovers accurate coordinates. We establish this progressively. We begin with a control experiment on uniformly sampled data, confirming that the additional alignment and averaging introduce no degradation relative to the baseline. We then demonstrate robustness under non-uniform sampling before concluding with a stress test under increasingly severe sampling bias.
 
\paragraph{Uniform sampling}

As a control, we confirm that the alignment and averaging steps introduce no distortion when the sampling is already uniform. We draw $250$ points uniformly from the sphere (\cref{fig:uniform-input}), add Gaussian noise of standard
deviation $0.05$, and compare the baseline against the subsampling method in two configurations: $20$ subsamples of size $75$ and $30$ of size $100$. All three recover the ground-truth azimuth and elevation comparably (\cref{fig:uniform-correlation}), with no appreciable difference between the two subsampling configurations. The additional subsampling, alignment, and averaging therefore introduce no measurable loss in coordinate quality under uniform sampling.
 
\begin{figure}[htbp]
    \centering
    
    \begin{subfigure}{0.35\textwidth}
        \centering
        \includegraphics[
            width=\linewidth,
            trim={0 1.1cm 0 1.1cm},
            clip
        ]{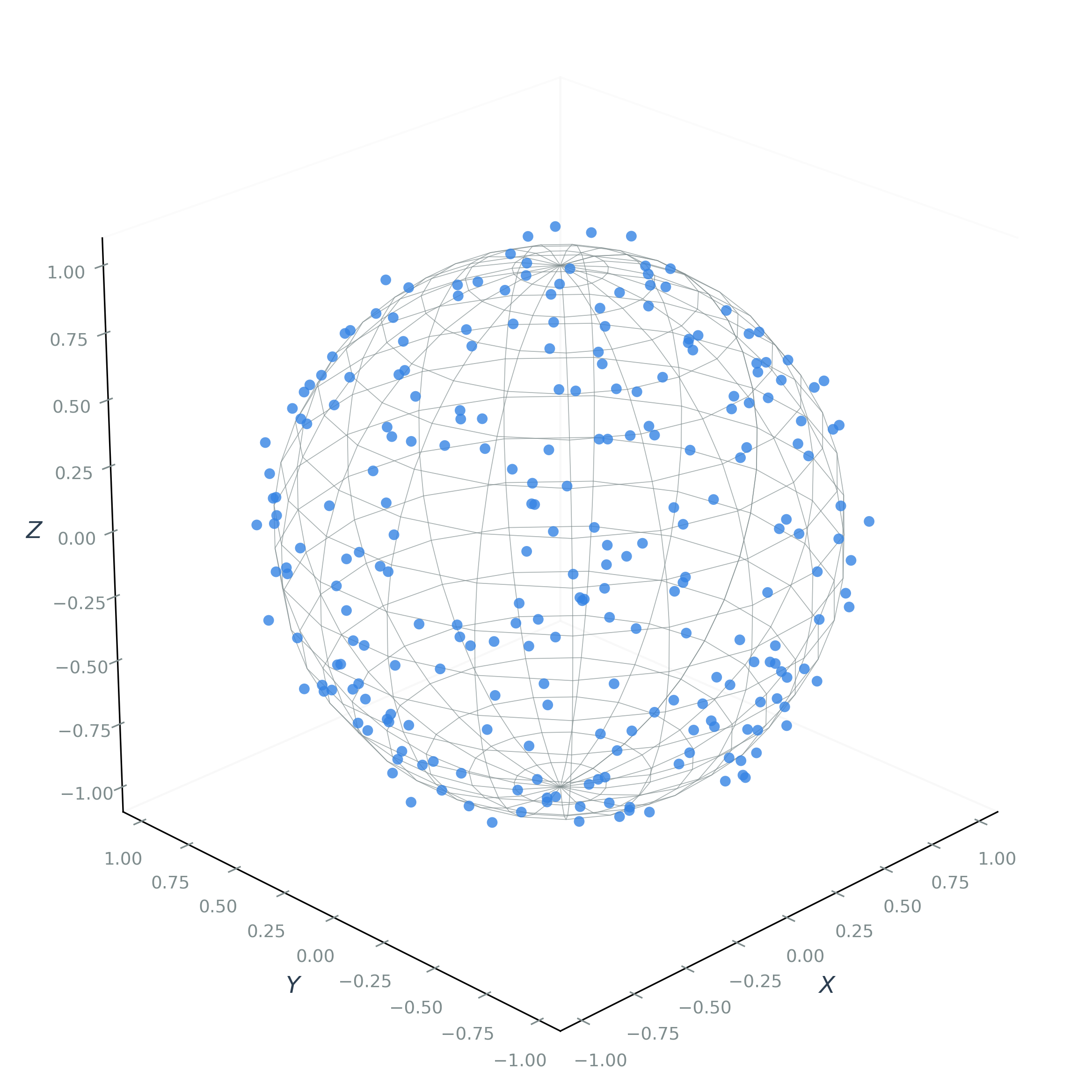}
        \caption{Input: Uniformly sampled point cloud on $S^2$.}
        \label{fig:uniform-input}
    \end{subfigure}
    
    \vspace{0.5em}
    
    \begin{subfigure}[b]{0.3\textwidth}
        \includegraphics[width=\textwidth]{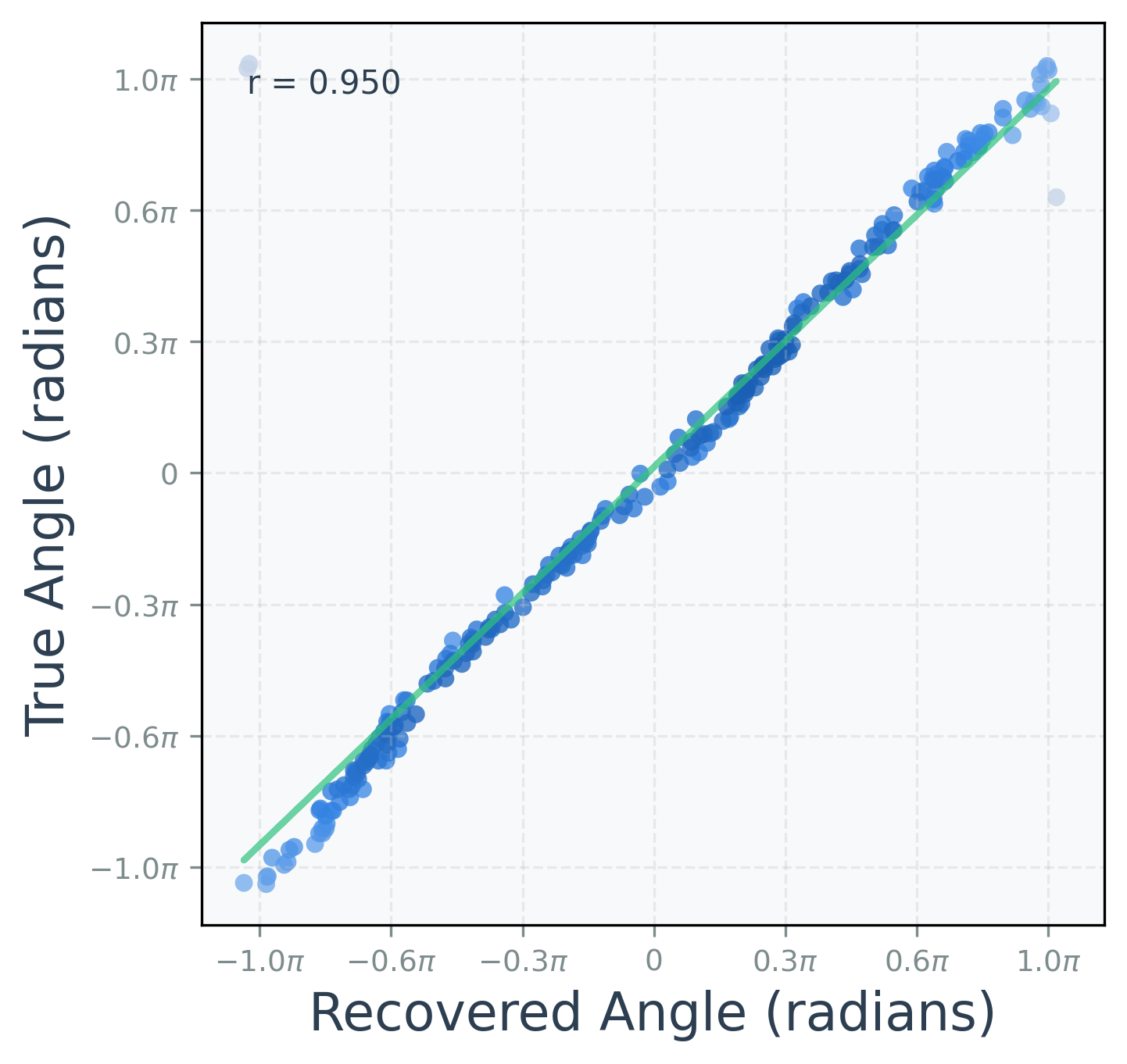}
        \caption{Baseline: azimuth}
    \end{subfigure}
    \hfill
    \begin{subfigure}[b]{0.3\textwidth}
        \includegraphics[width=\textwidth]{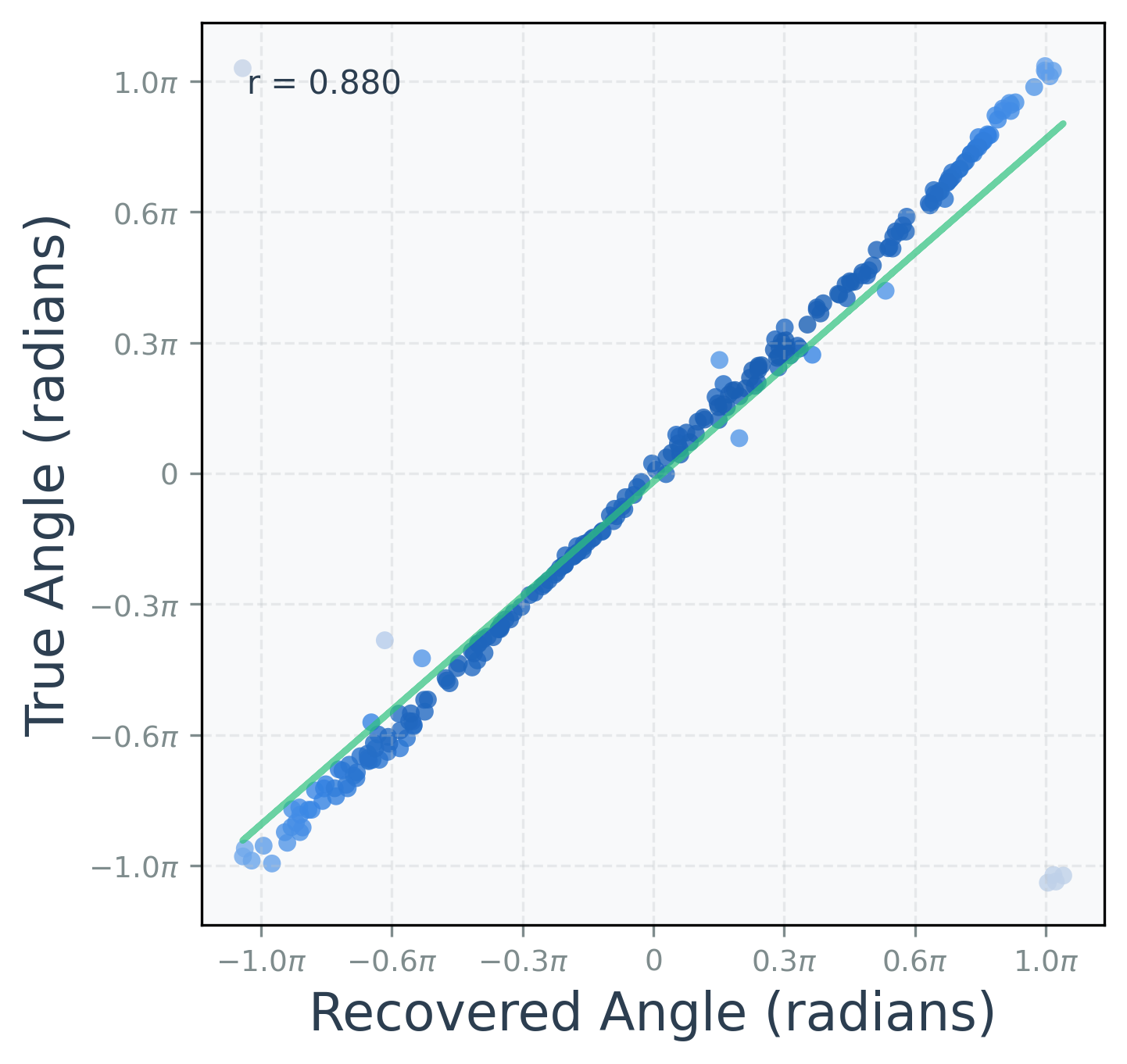}
        \caption{Subsampling $20\times75$: azimuth}
    \end{subfigure}
    \hfill
    \begin{subfigure}[b]{0.3\textwidth}
        \includegraphics[width=\textwidth]{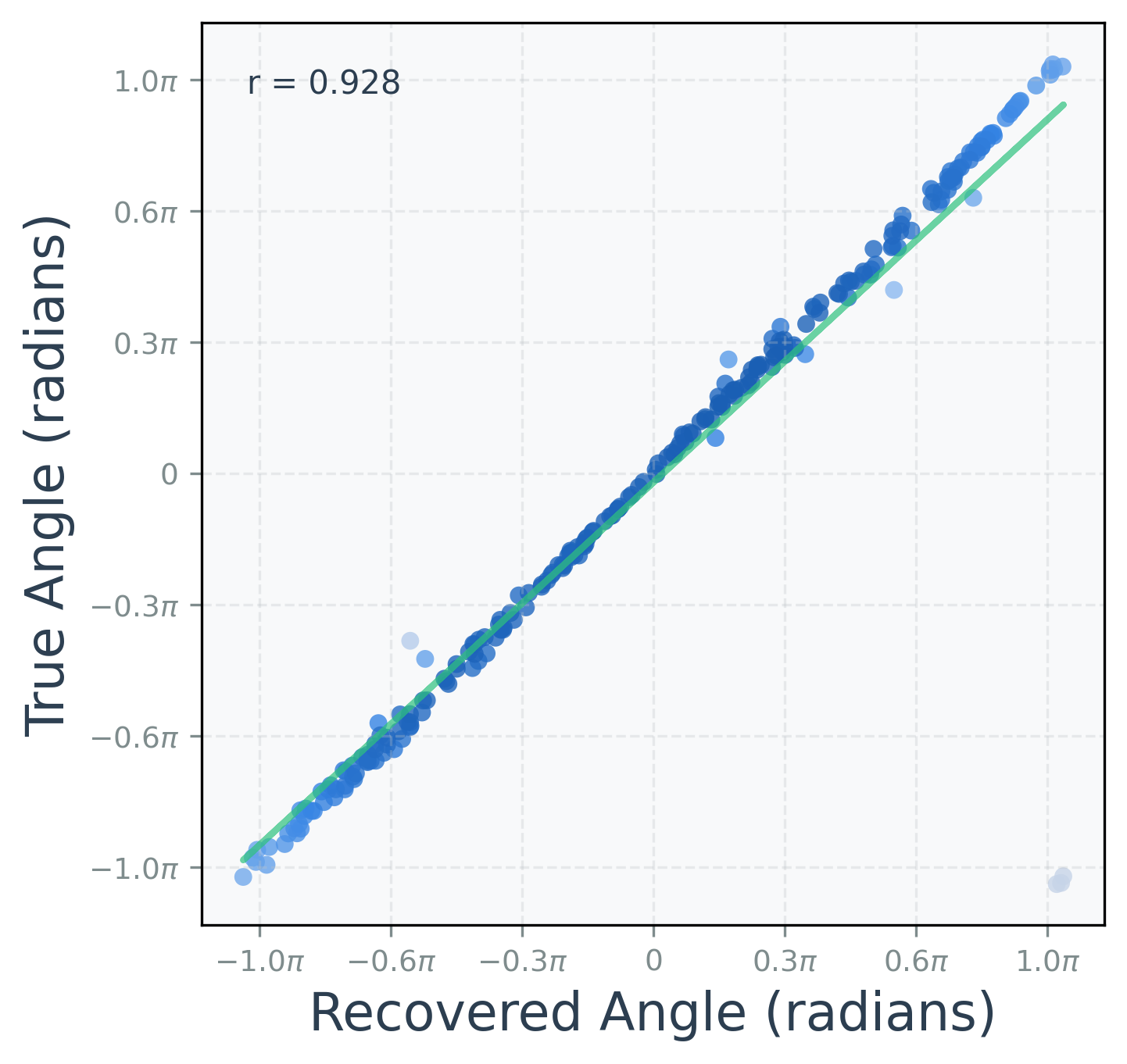}
        \caption{Subsampling $30\times100$: azimuth}
    \end{subfigure}
    
    \begin{subfigure}[b]{0.3\textwidth}
        \includegraphics[width=\textwidth]{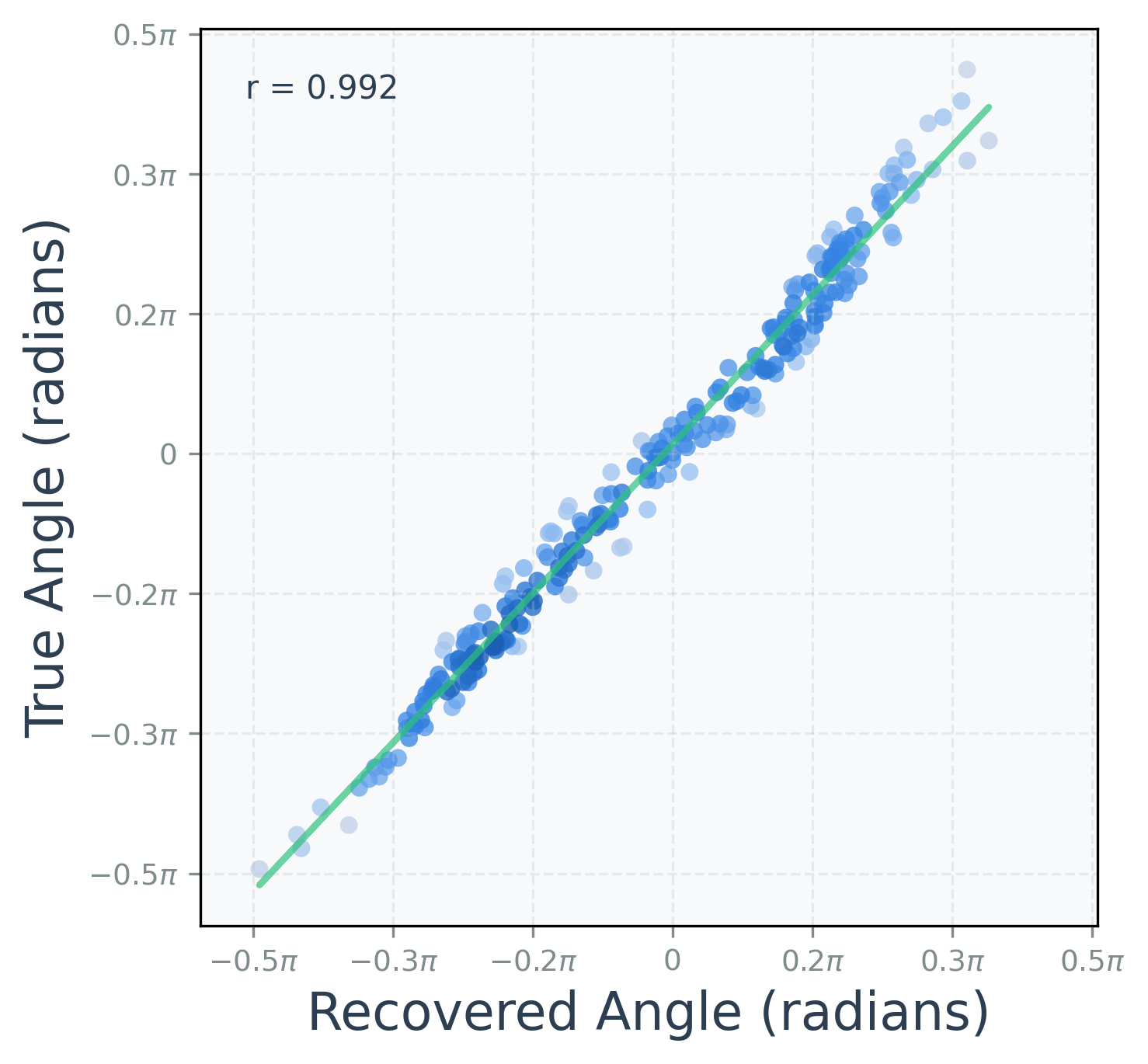}
        \caption{Baseline: elevation}
    \end{subfigure}
    \hfill
    \begin{subfigure}[b]{0.3\textwidth}
        \includegraphics[width=\textwidth]{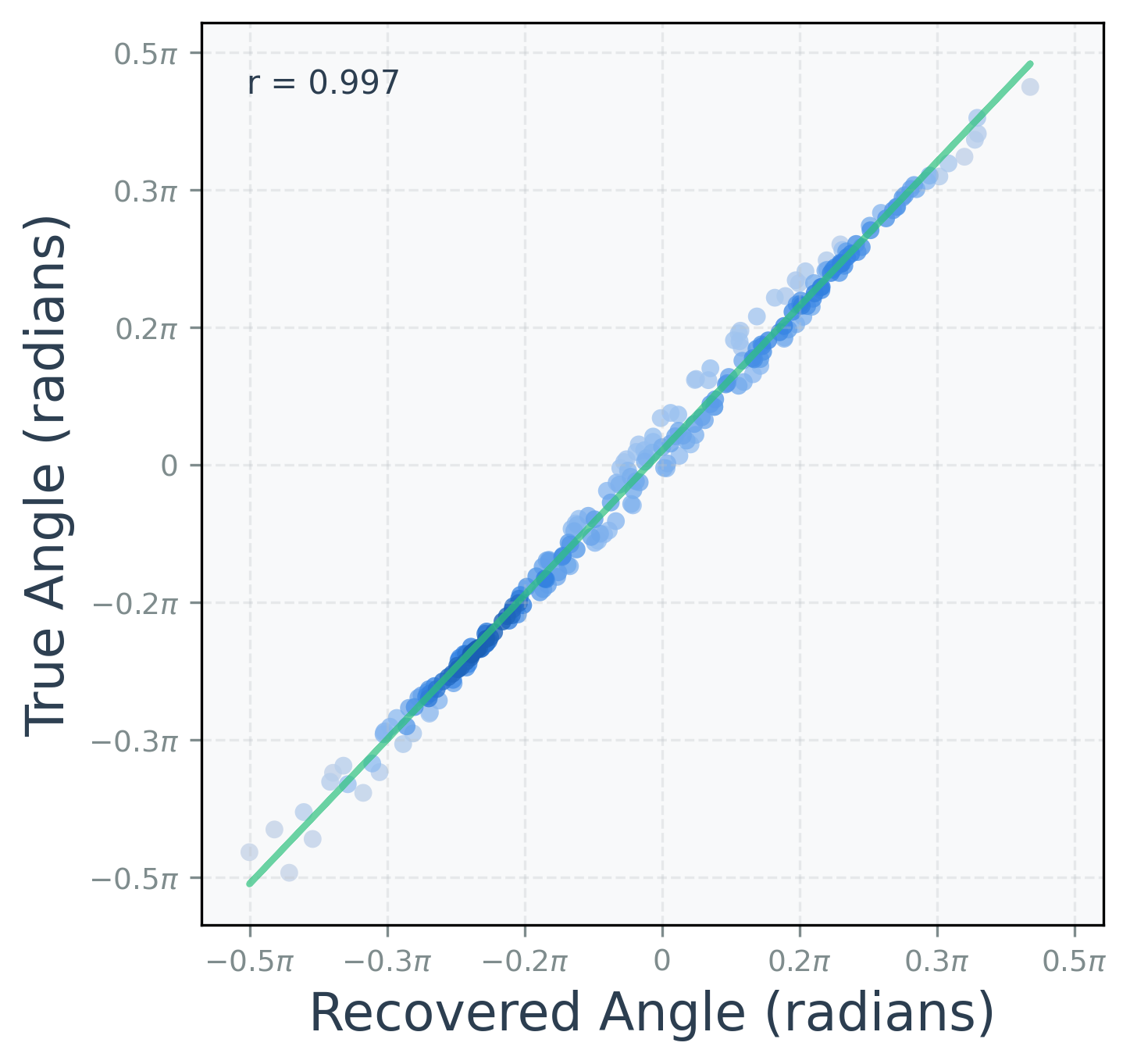}
        \caption{Subsampling $20\times75$: elevation}
    \end{subfigure}
    \hfill
    \begin{subfigure}[b]{0.3\textwidth}
        \includegraphics[width=\textwidth]{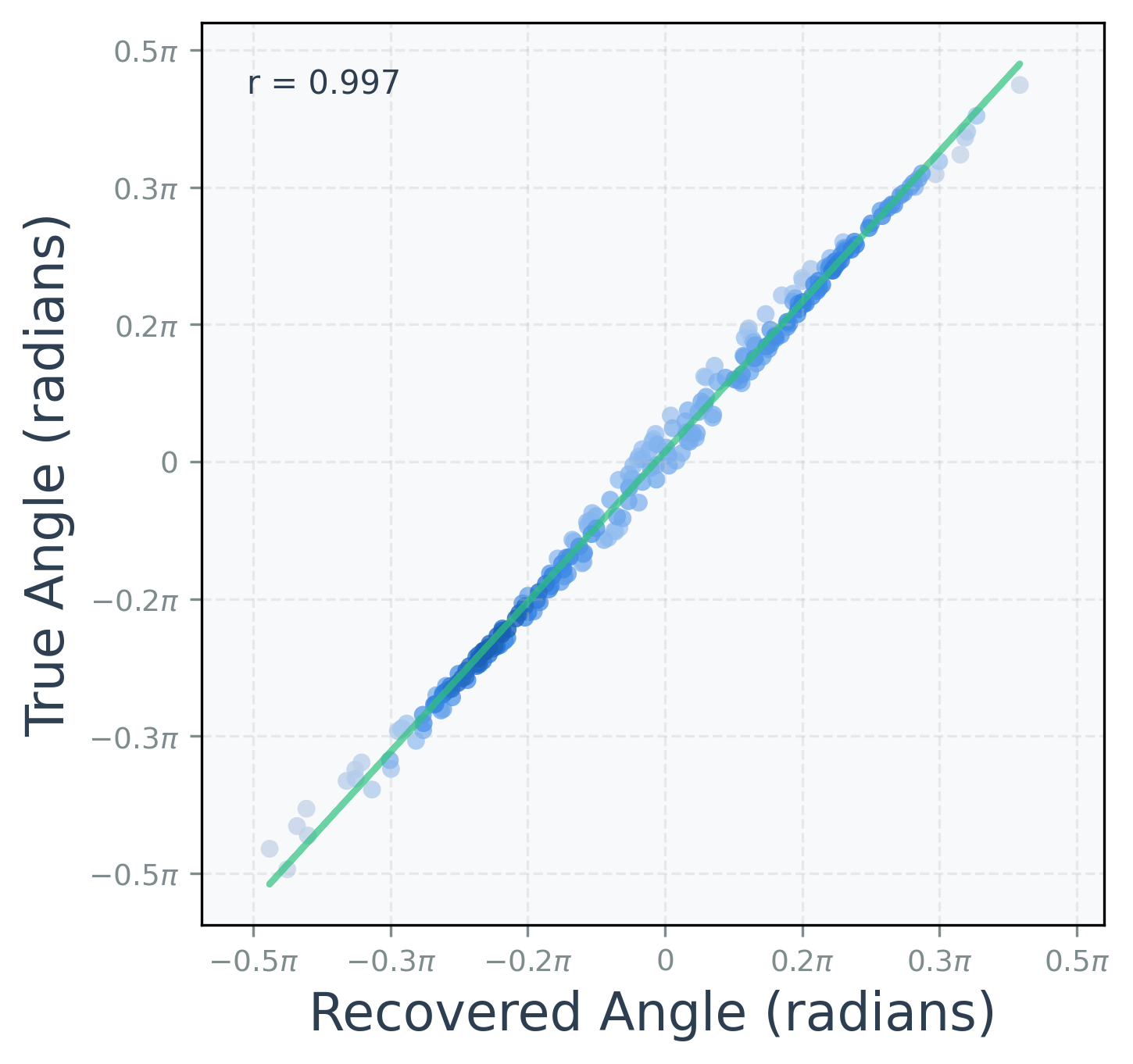}
        \caption{Subsampling $30\times100$: elevation}
    \end{subfigure}
 
    \caption{Uniform sampling control experiment. Top: uniform sample over $S^2$. Middle and bottom rows: recovered azimuth and elevation angles (respectively), each compared against the ground truth. Columns in these rows correspond to the baseline method and the two subsampling configurations.  The subsampling method performs comparably to the baseline in both angles, in both configurations.}
    \label{fig:uniform-correlation}
\end{figure}

\paragraph{Robustness to non-uniform sampling}
This is our main experiment. We sample $300$ points from a von Mises--Fisher distribution at the north pole with concentration $\kappa = 1.3$ (\cref{fig:nonuniform-input}), which has density concentrated near the pole and
sparse near the equator, and apply the baseline and the subsampling method with $20$ subsamples of size $75$ to the same data. In this case, we observe a sharp separation in behaviour: the baseline fails to recover the spherical angles, while the subsampling method stays accurate in both angles (\cref{fig:nonuniform-embeddings}). This phenomenon is explained by the normalization step used in the classical spherical coordinate construction. The observed distortion is a direct consequence of the definition. Recall that the baseline's energy functional does not have a unique minimizer: its minimizers form an equivalence class under the conformal (M\"obius) self-maps of $S^2$, which preserve the energy. To single out one representative, the baseline method requires the recovered map to have zero center of mass, that is, the image points $\Phi(x_i)$ to average to the origin of $\mathbb{R}^3$ (\cref{sec:spherical_coords}). For uniformly sampled data this is essentially free: the image is spread evenly over the sphere and is already balanced about the origin, which is why the control experiment above is unaffected. A von Mises--Fisher sample, however, concentrates around its mean direction (here, the north pole) so the recovered image concentrates there too and its center of mass points toward the pole, far from the origin. Enforcing the constraint then drags this center of mass back to the origin through a non-rigid M\"obius transformation, and because such maps are conformal but not isometric, they stretch and compress the sphere unevenly; the resulting distortion is the bias visible in \cref{fig:nonuniform-embeddings}. 
The subsampling method, however, avoids this because each subsample is approximately uniform, so on each one the constraint is essentially satisfied already and the per-subsample coordinates are unbiased. This is preserved through the extension and alignment to combine them into a single global coordinate.

\begin{figure}[htbp]
    \centering

    \begin{subfigure}{0.35\textwidth}
        \centering
        \includegraphics[
            width=\linewidth,
            trim={0 1.2cm 0 1.2cm},
            clip
        ]{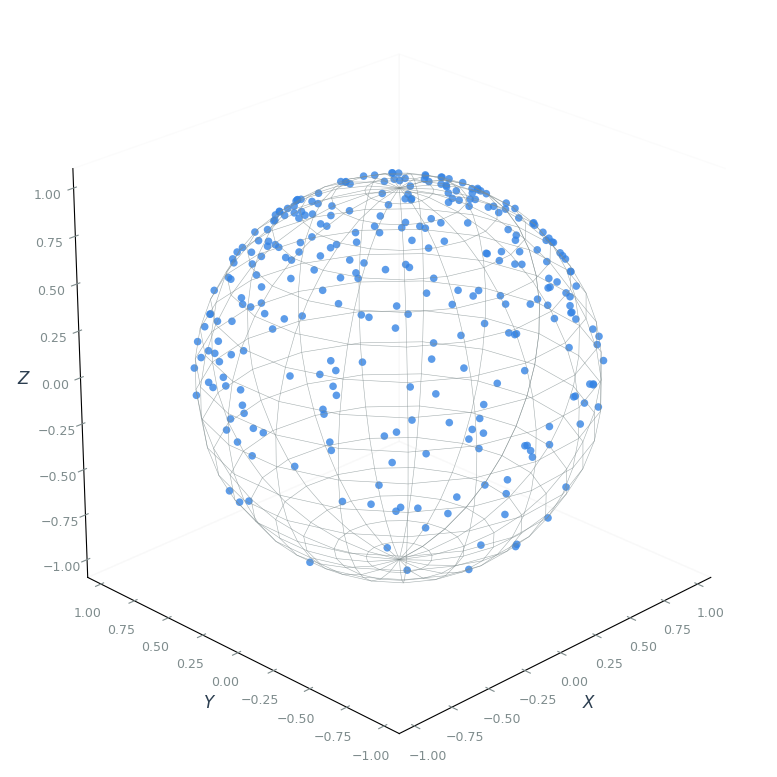}
        \caption{Input: von Mises--Fisher sample centered at $(0,0,1)$.}
        \label{fig:nonuniform-input}
    \end{subfigure}
    
    \vspace{0.5em}
    
    \begin{subfigure}[b]{0.3\textwidth}
        \centering
        \includegraphics[width=\textwidth]{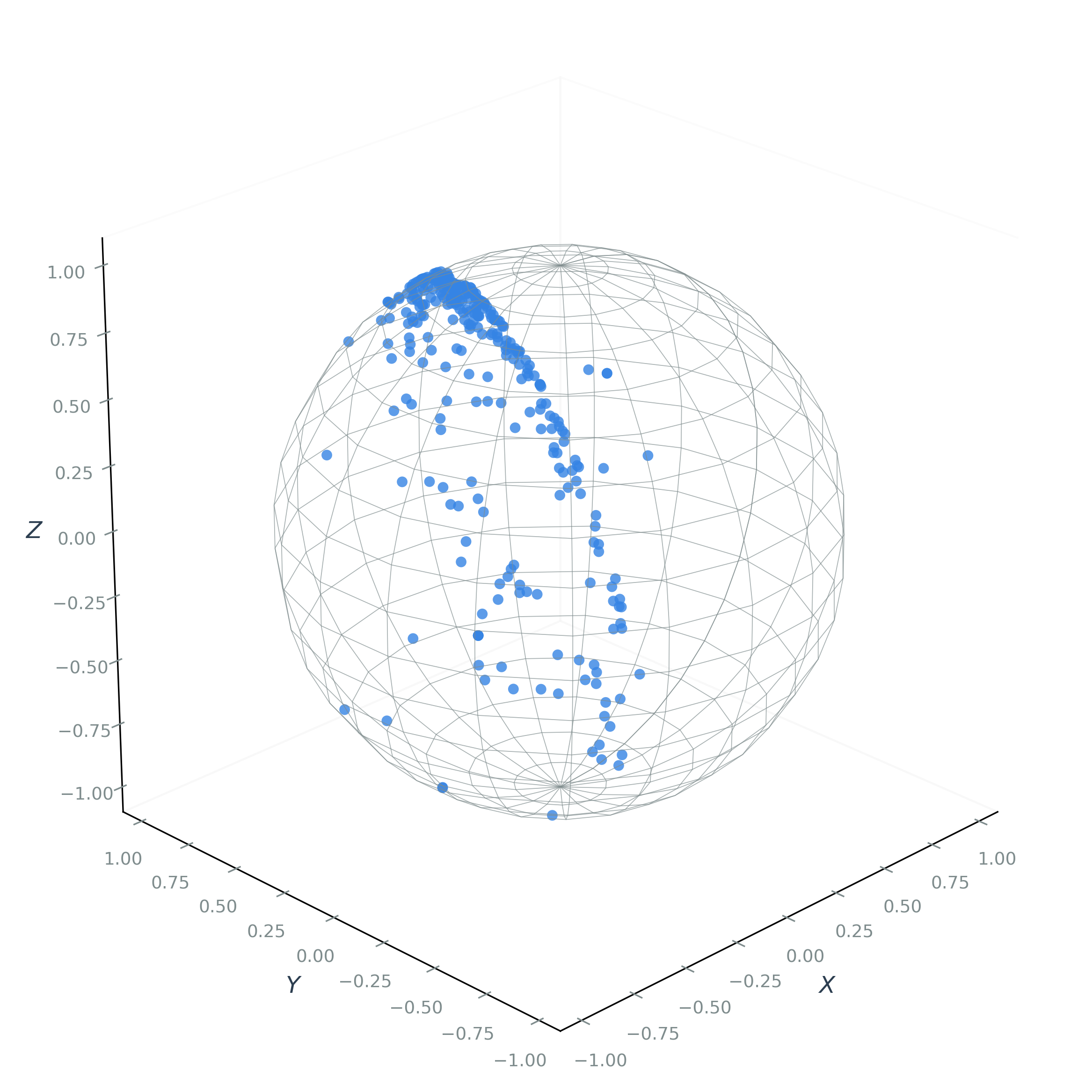}
        \caption{Baseline: embedding}
    \end{subfigure}
    \hfill
    \begin{subfigure}[b]{0.3\textwidth}
        \centering
        \includegraphics[width=\textwidth]{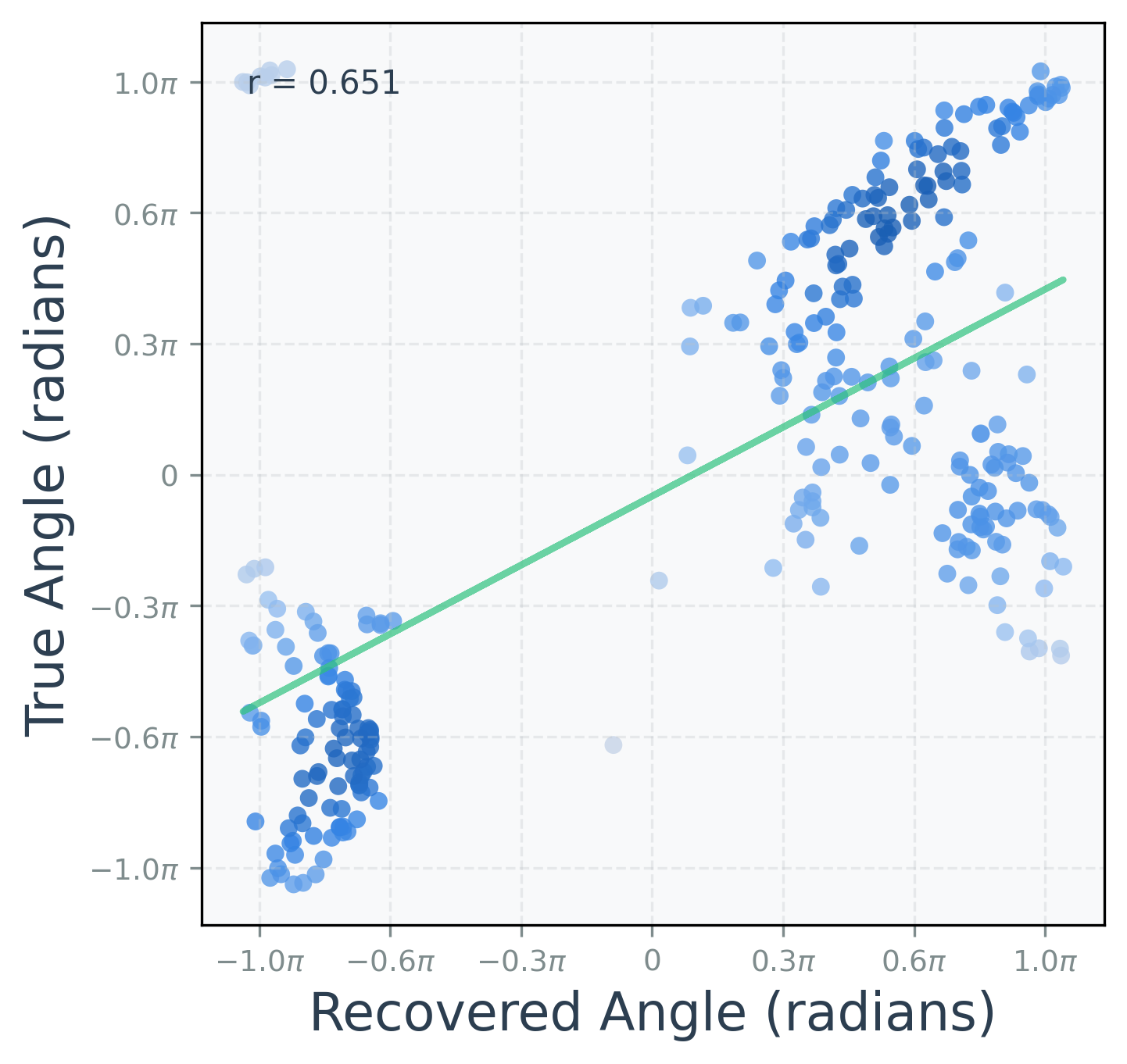}
        \caption{Baseline: azimuth}
    \end{subfigure}
    \hfill
    \begin{subfigure}[b]{0.3\textwidth}
        \centering
        \includegraphics[width=\textwidth]{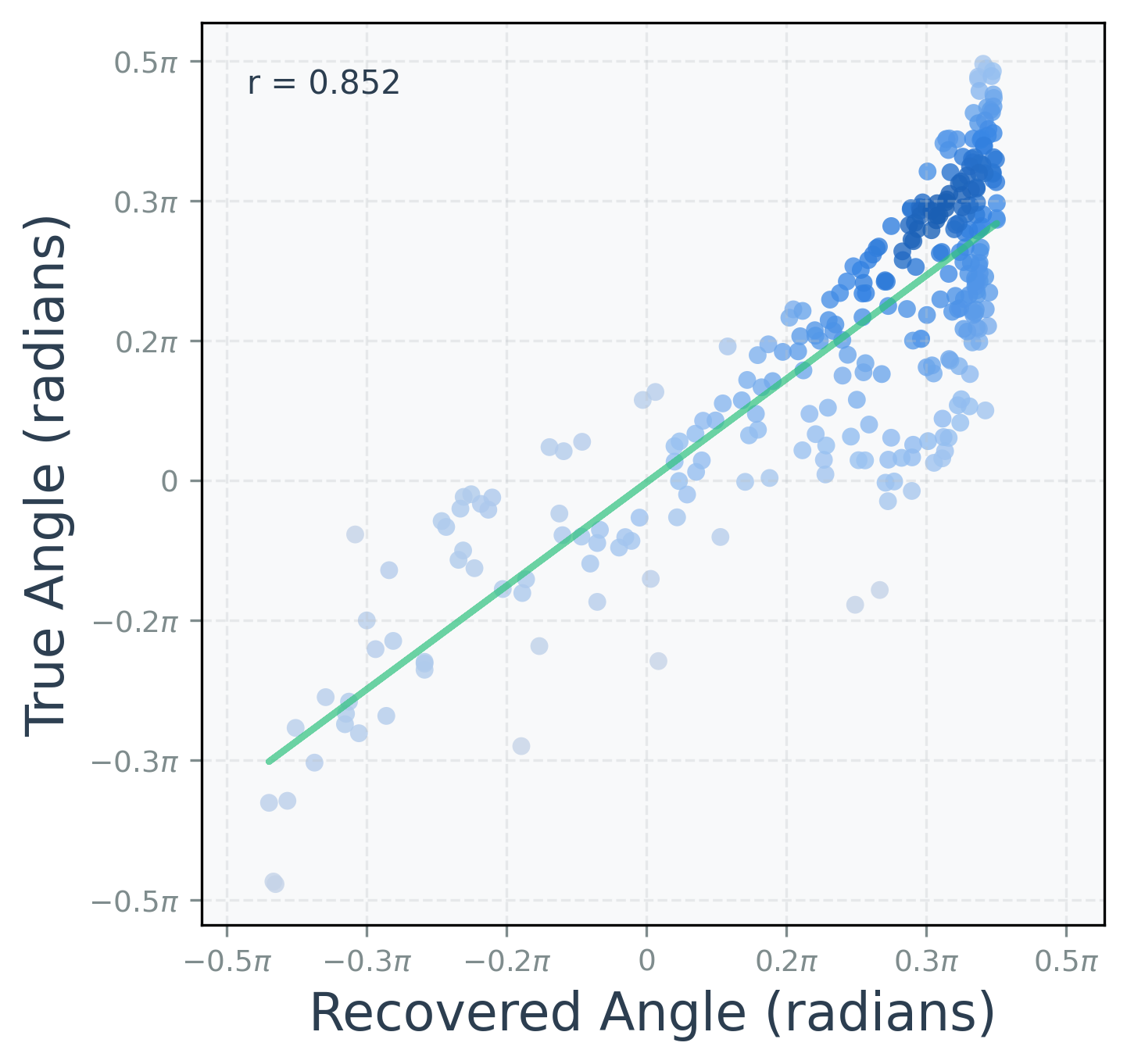}
        \caption{Baseline: elevation}
    \end{subfigure}

    \begin{subfigure}[b]{0.3\textwidth}
        \centering
        \includegraphics[width=\textwidth]{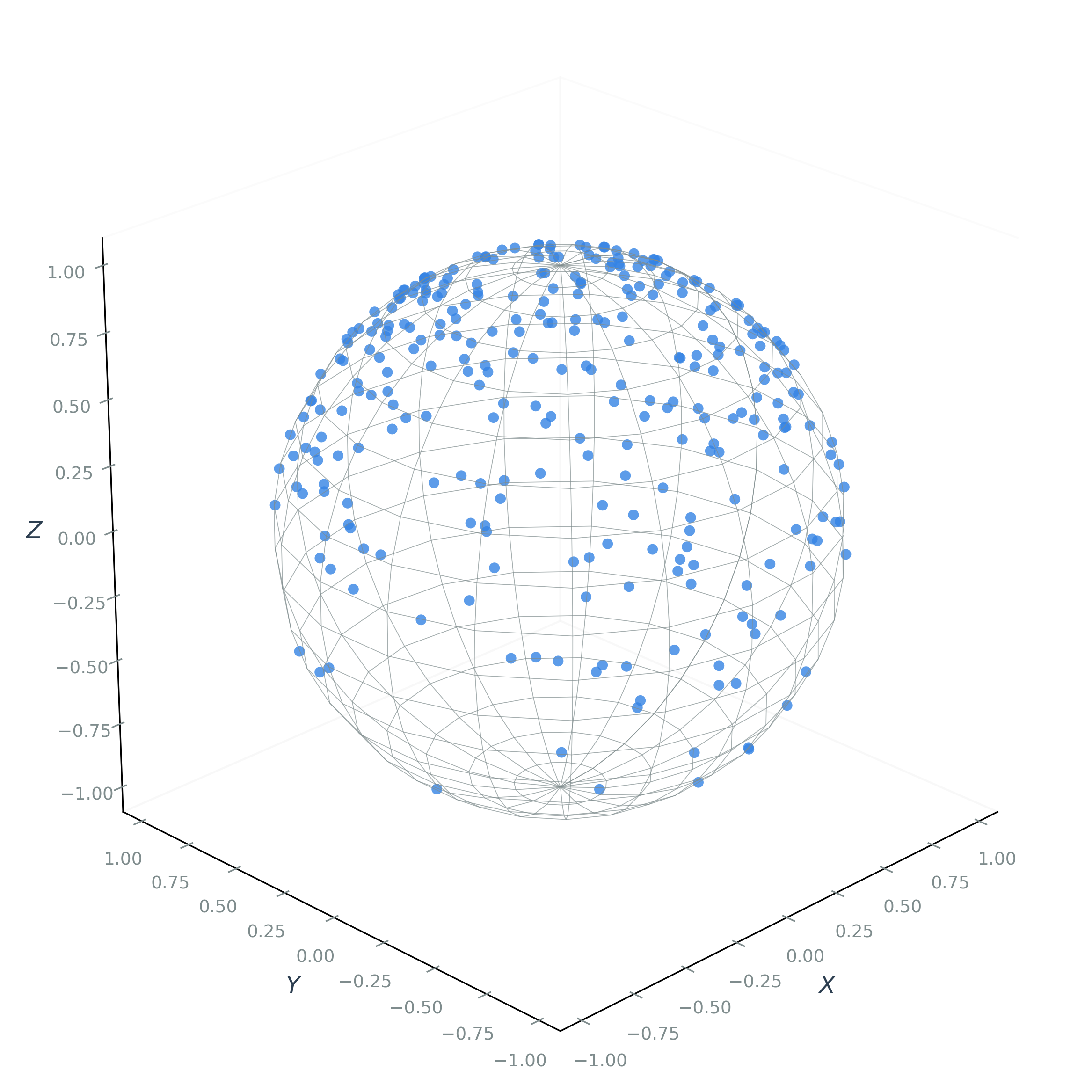}
        \caption{Subsampling: embedding}
    \end{subfigure}
    \hfill
    \begin{subfigure}[b]{0.3\textwidth}
        \centering
        \includegraphics[width=\textwidth]{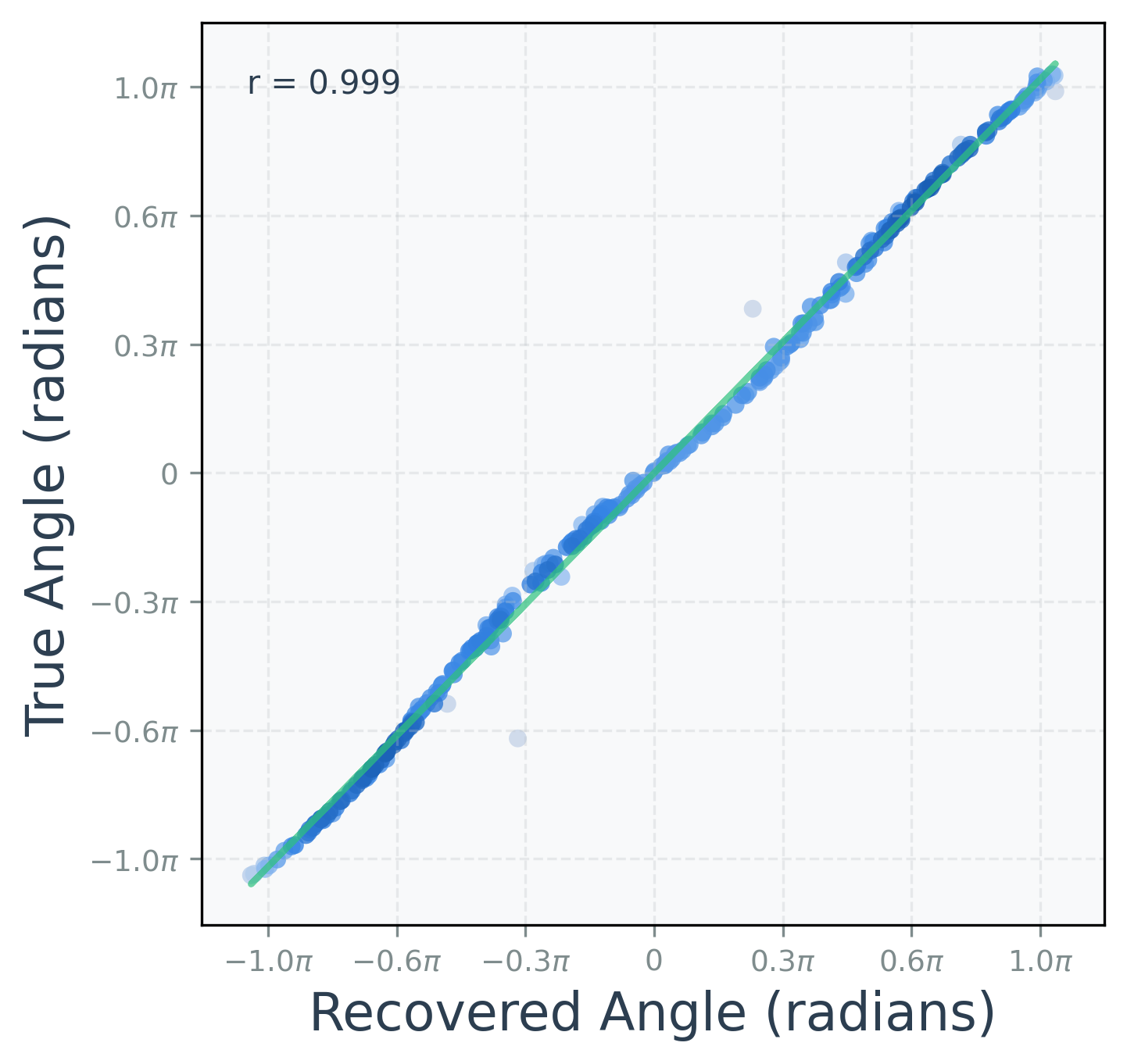}
        \caption{Subsampling: azimuth}
    \end{subfigure}
    \hfill
    \begin{subfigure}[b]{0.3\textwidth}
        \centering
        \includegraphics[width=\textwidth]{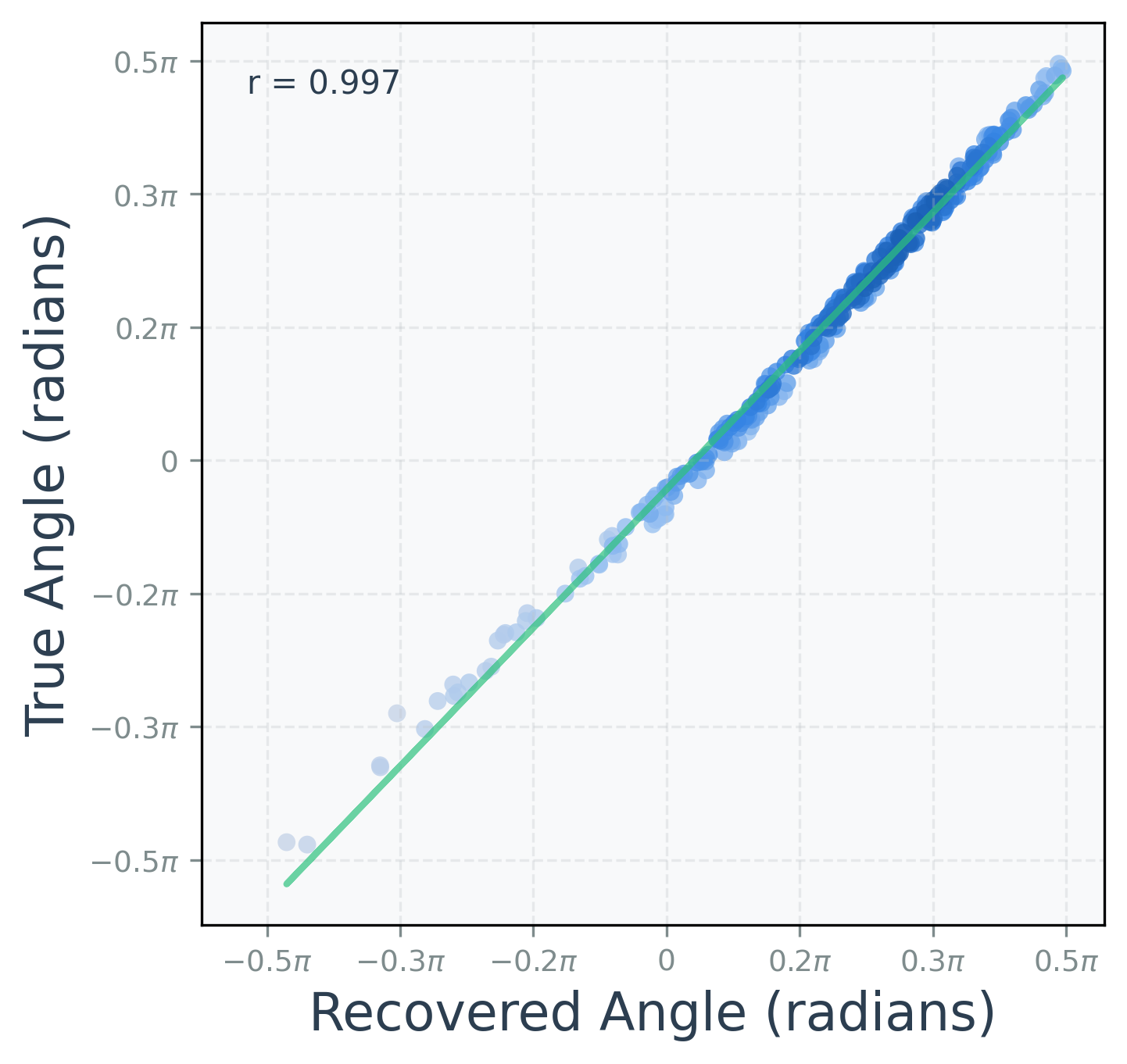}
        \caption{Subsampling: elevation}
    \end{subfigure}

    \caption{Embeddings and angle recovery on the non-uniform dataset. Top:
    input sampled from a von Mises--Fisher distribution centered at $(0,0,1)$.
    Middle: baseline spherical coordinates. Bottom: the proposed subsampling
    method. The baseline suffers from strong distortion caused by the sampling
    bias, while the subsampling method maintains accurate recovery of both
    spherical angles.}
    \label{fig:nonuniform-embeddings}
\end{figure}
 
\paragraph{Severe sampling bias}
Finally, we push the imbalance further to test how far the robustness extends. We generate three clouds of $1{,}500$ points and apply the subsampling method with $30$ subsamples of size $120$ to each: a von Mises--Fisher distribution with $\kappa = 2$ clustered at the pole, a distribution concentrated in an equatorial band, and a bimodal one split between the two poles. The recovered angles track the ground truth in all three cases (\cref{fig:strong-bias}), so the method preserves spherical structure even under severe density variation: regardless of how the sampling density is distributed, rejection sampling restores approximate uniformity within each subsample, and the per-subsample coordinates therefore remain unbiased.
 
\begin{figure}[htbp]
    \centering
    \begin{subfigure}[b]{0.32\textwidth}
        \includegraphics[width=\textwidth]{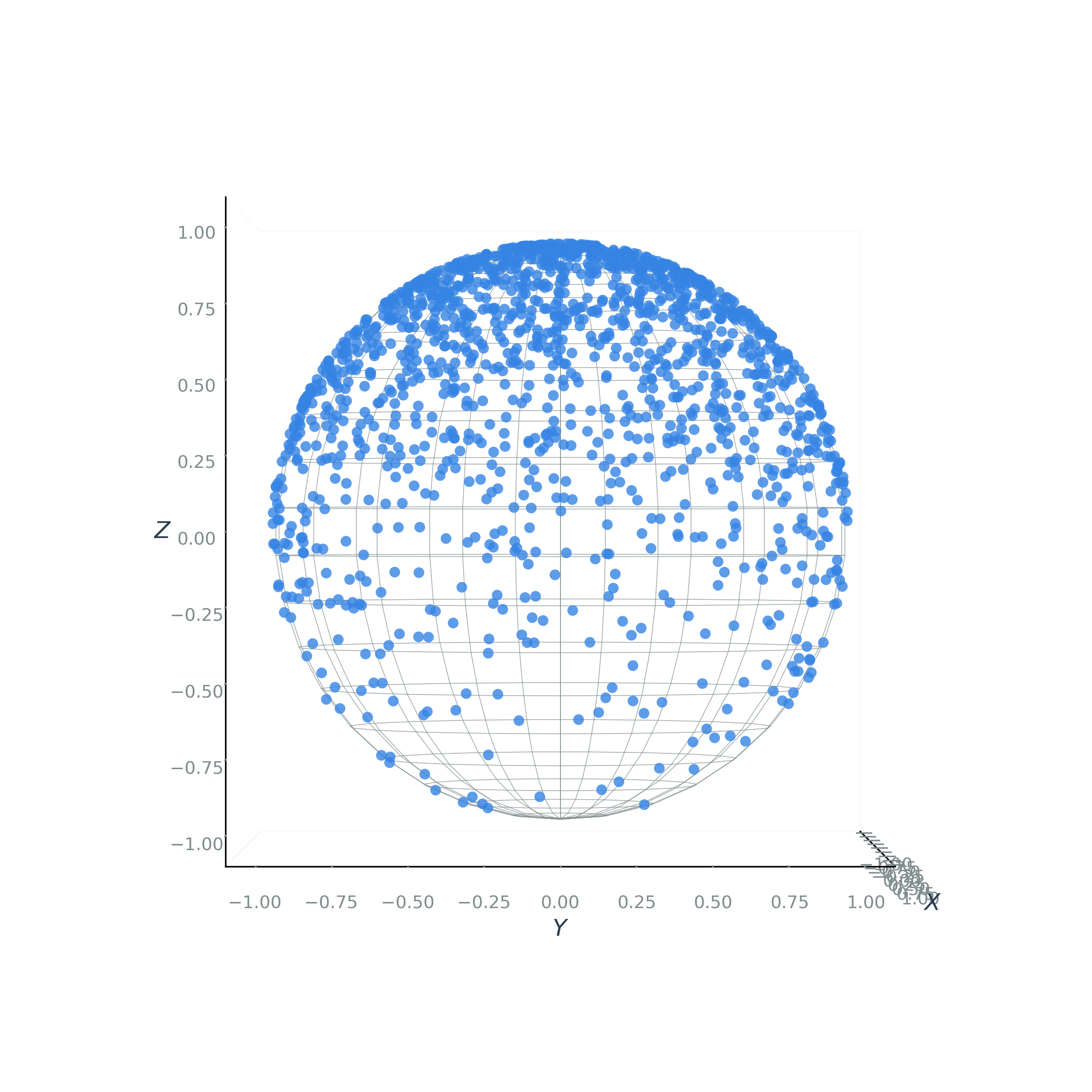}
        \caption{VMF $\kappa=2$: input}
    \end{subfigure}
    \hfill
    \begin{subfigure}[b]{0.32\textwidth}
        \includegraphics[width=\textwidth]{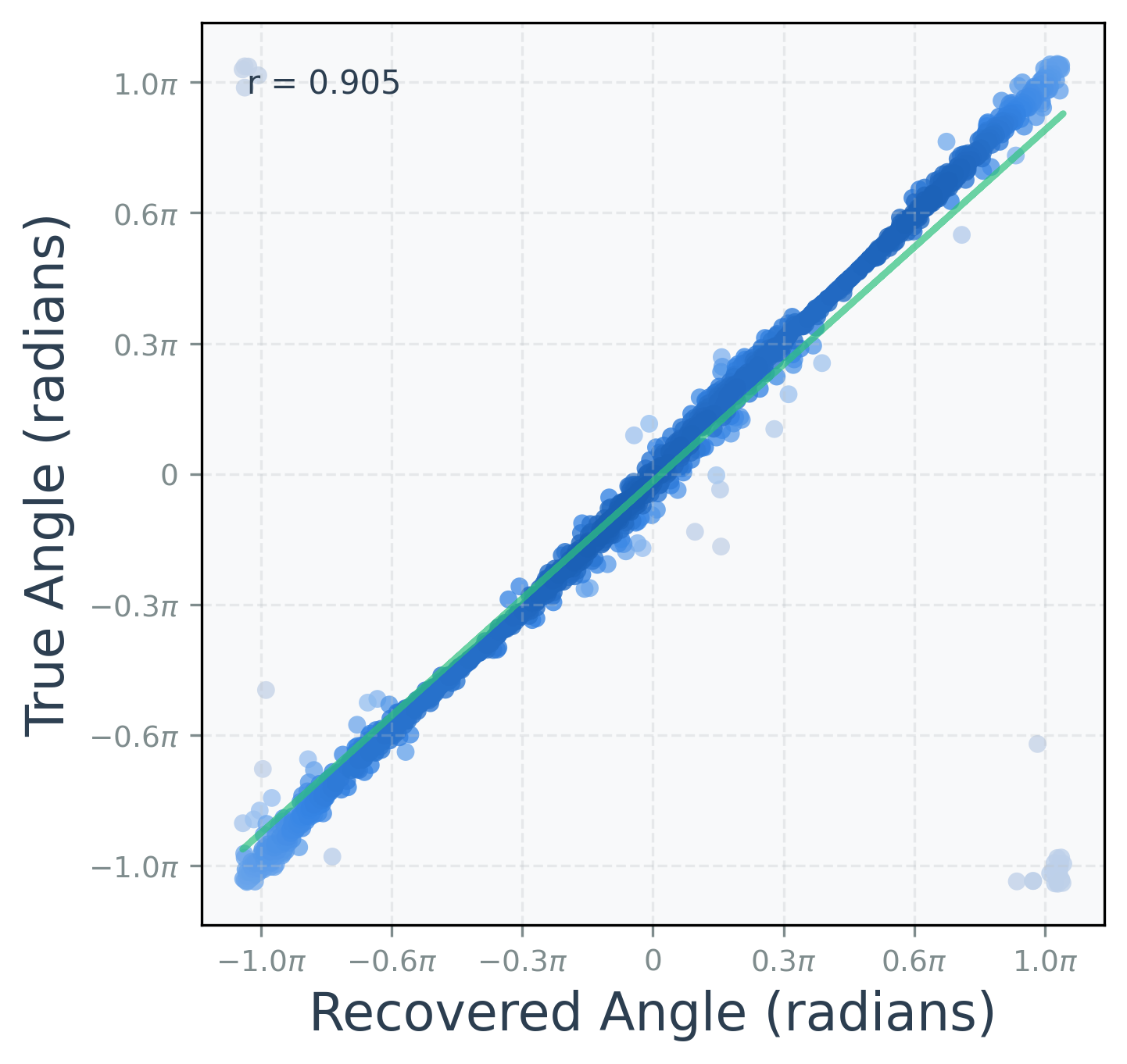}
        \caption{VMF $\kappa=2$: azimuth}
    \end{subfigure}
    \hfill
    \begin{subfigure}[b]{0.32\textwidth}
        \includegraphics[width=\textwidth]{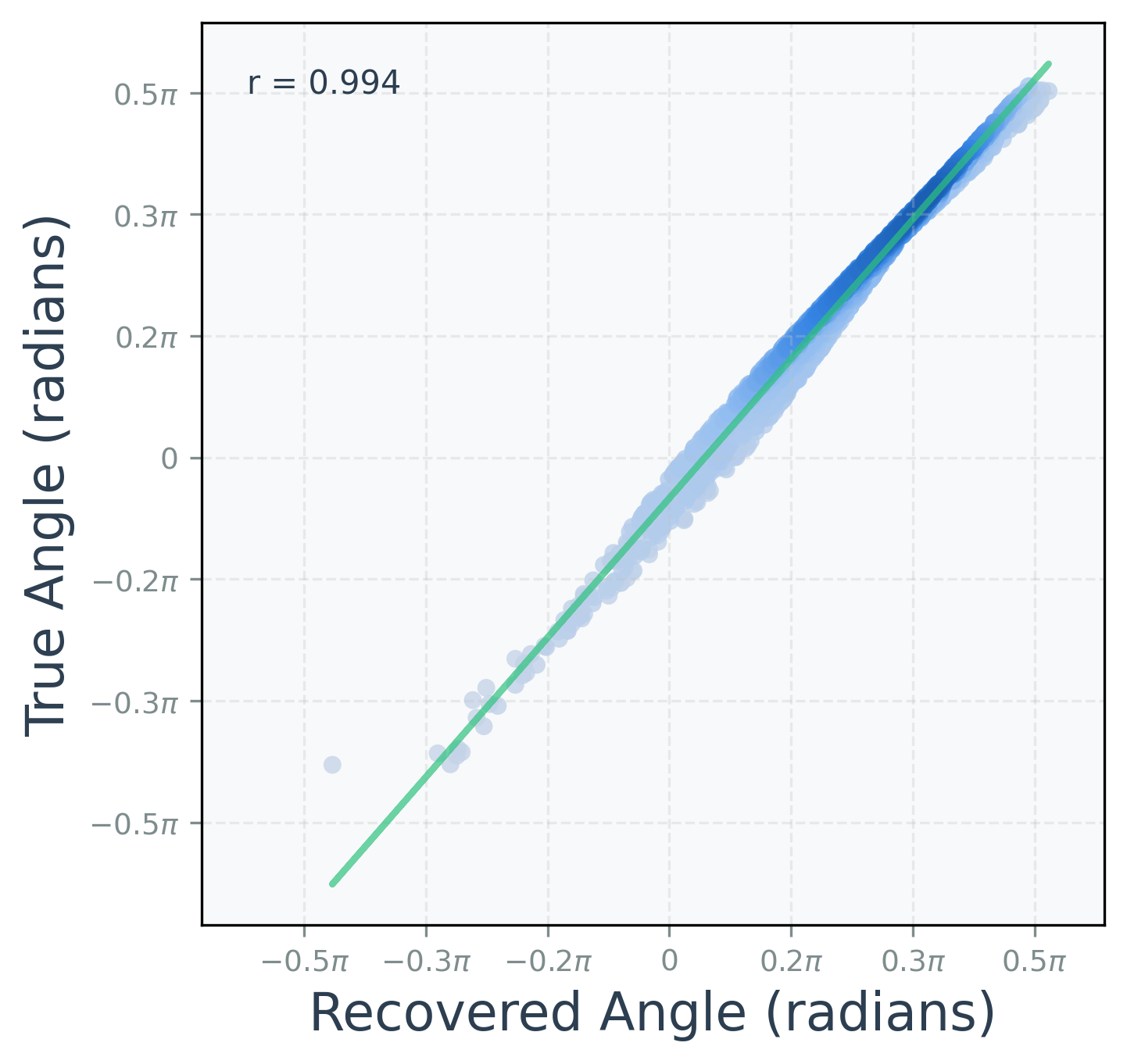}
        \caption{VMF $\kappa=2$: elevation}
    \end{subfigure}
 
    \begin{subfigure}[b]{0.32\textwidth}
        \includegraphics[width=\textwidth]{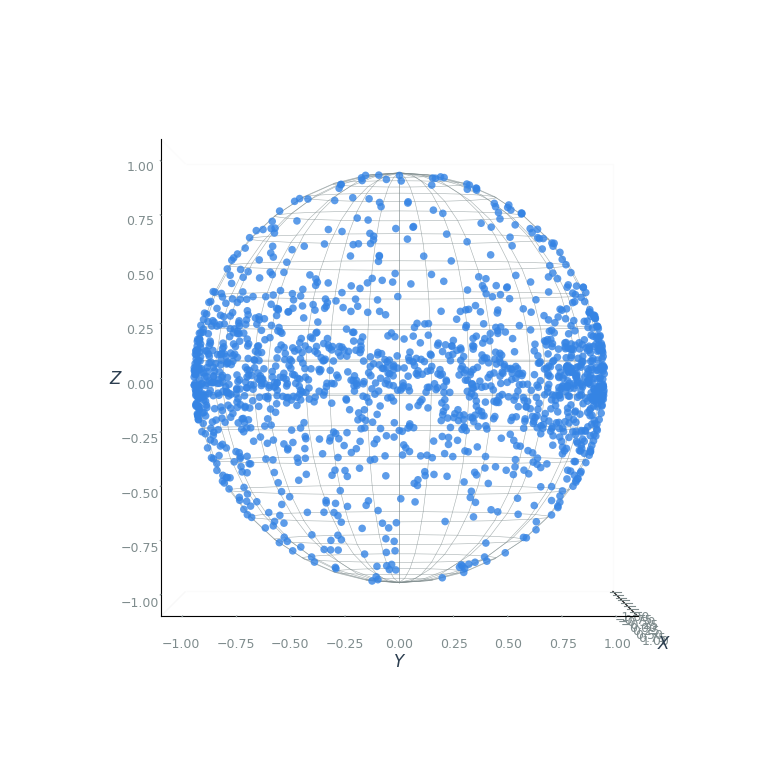}
        \caption{Equator: input}
    \end{subfigure}
    \hfill
    \begin{subfigure}[b]{0.32\textwidth}
        \includegraphics[width=\textwidth]{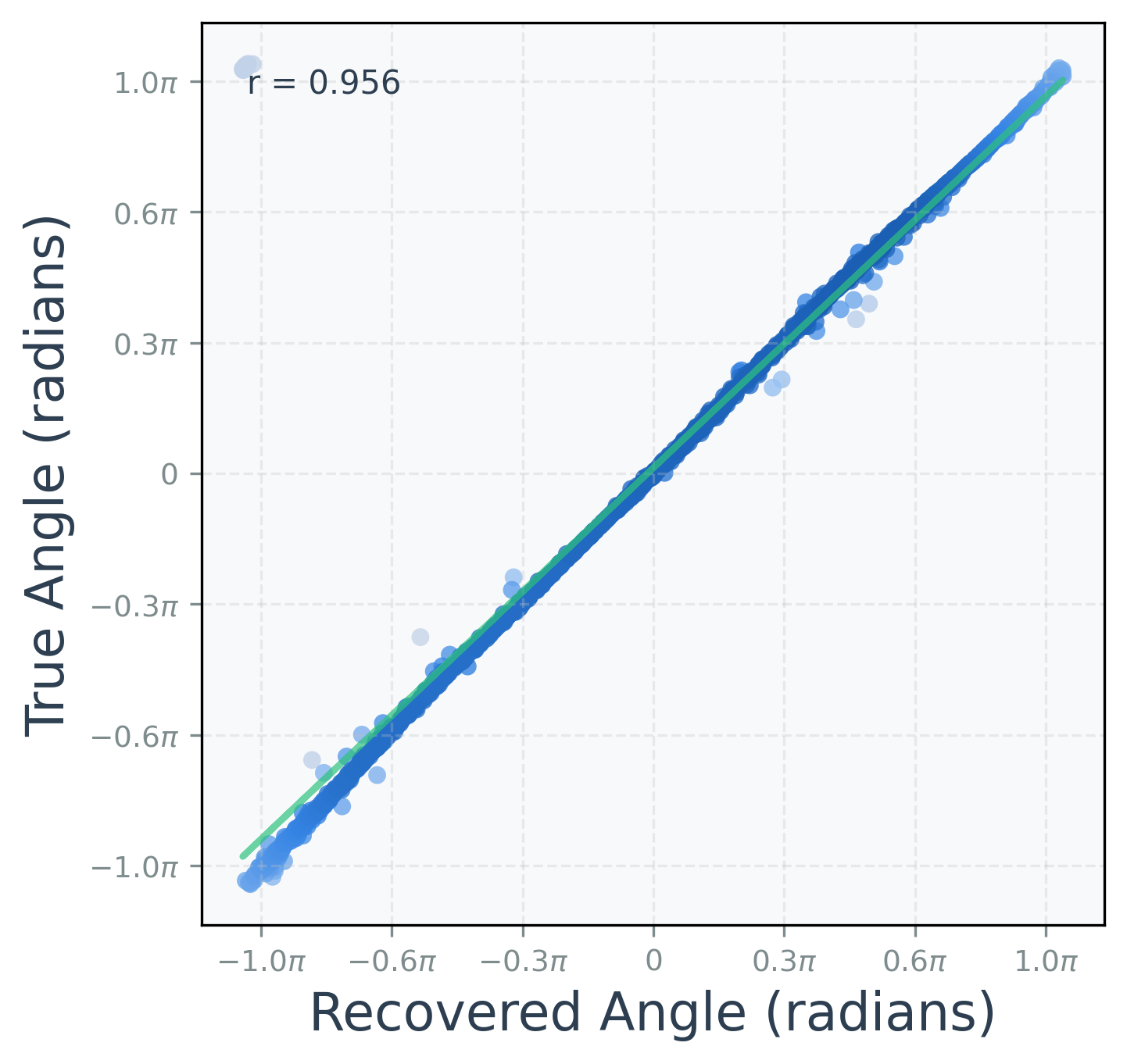}
        \caption{Equator: azimuth}
    \end{subfigure}
    \hfill
    \begin{subfigure}[b]{0.32\textwidth}
        \includegraphics[width=\textwidth]{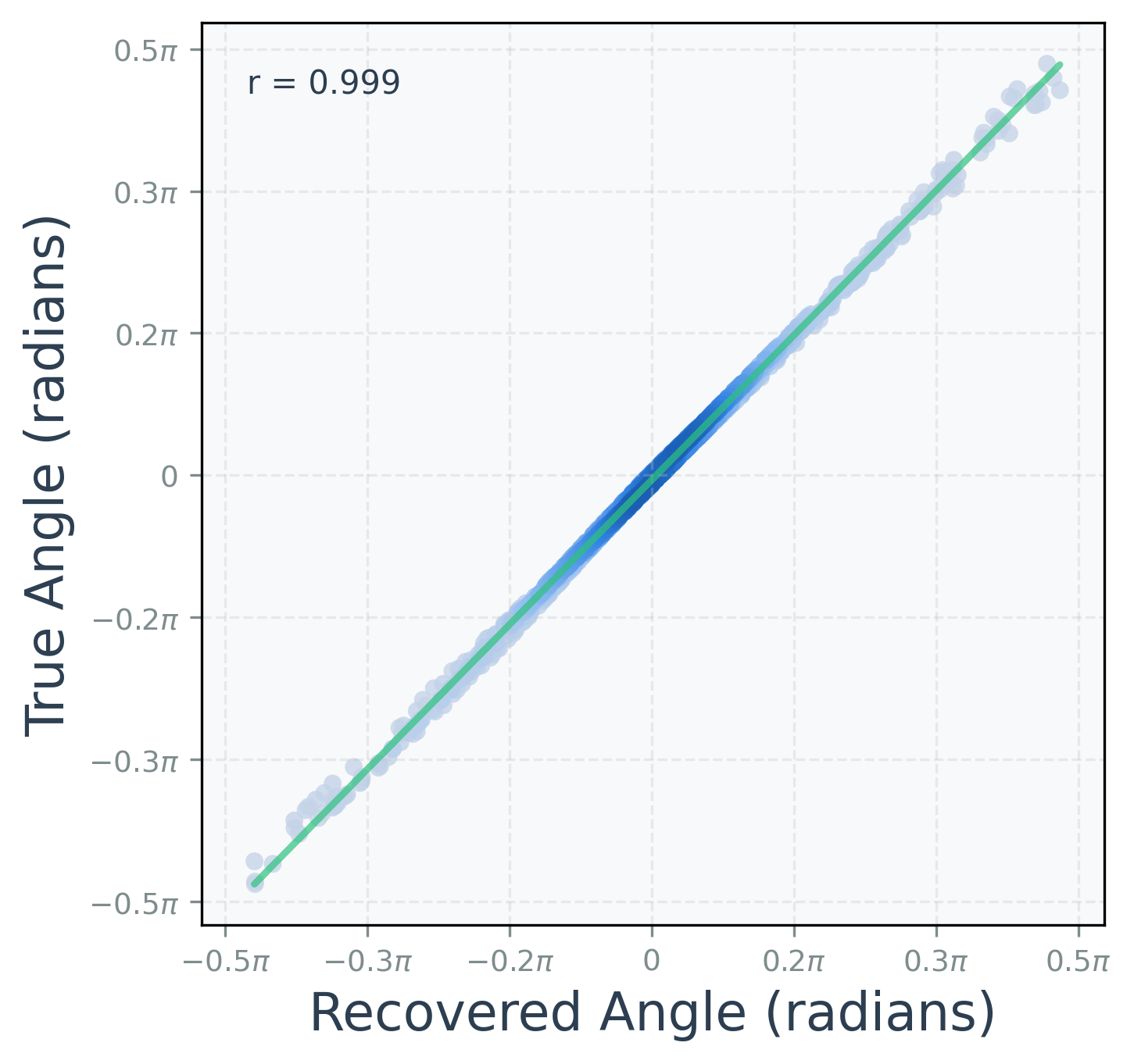}
        \caption{Equator: elevation}
    \end{subfigure}
 
    \begin{subfigure}[b]{0.32\textwidth}
        \includegraphics[width=\textwidth]{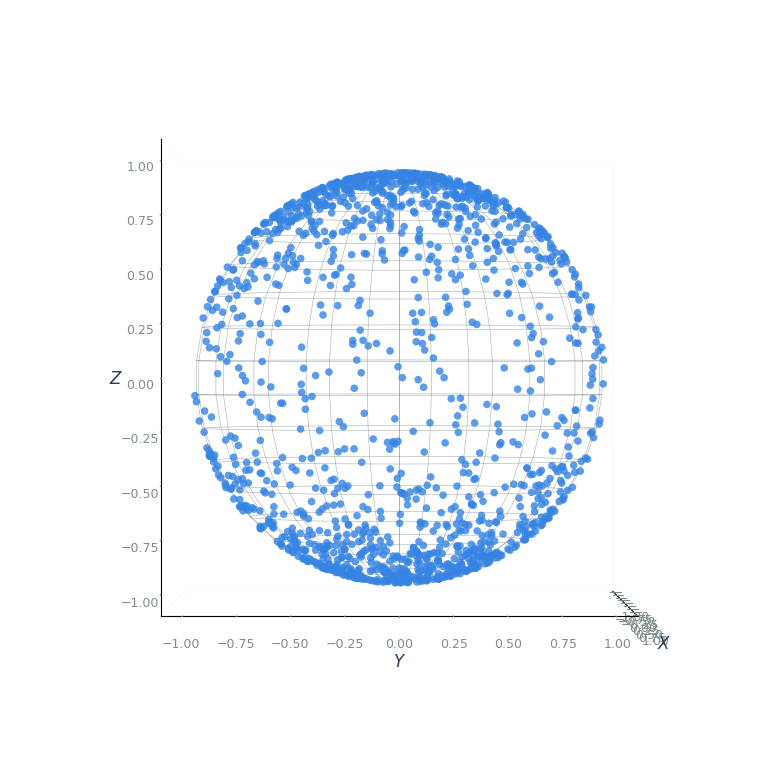}
        \caption{Poles: input}
    \end{subfigure}
    \hfill
    \begin{subfigure}[b]{0.32\textwidth}
        \includegraphics[width=\textwidth]{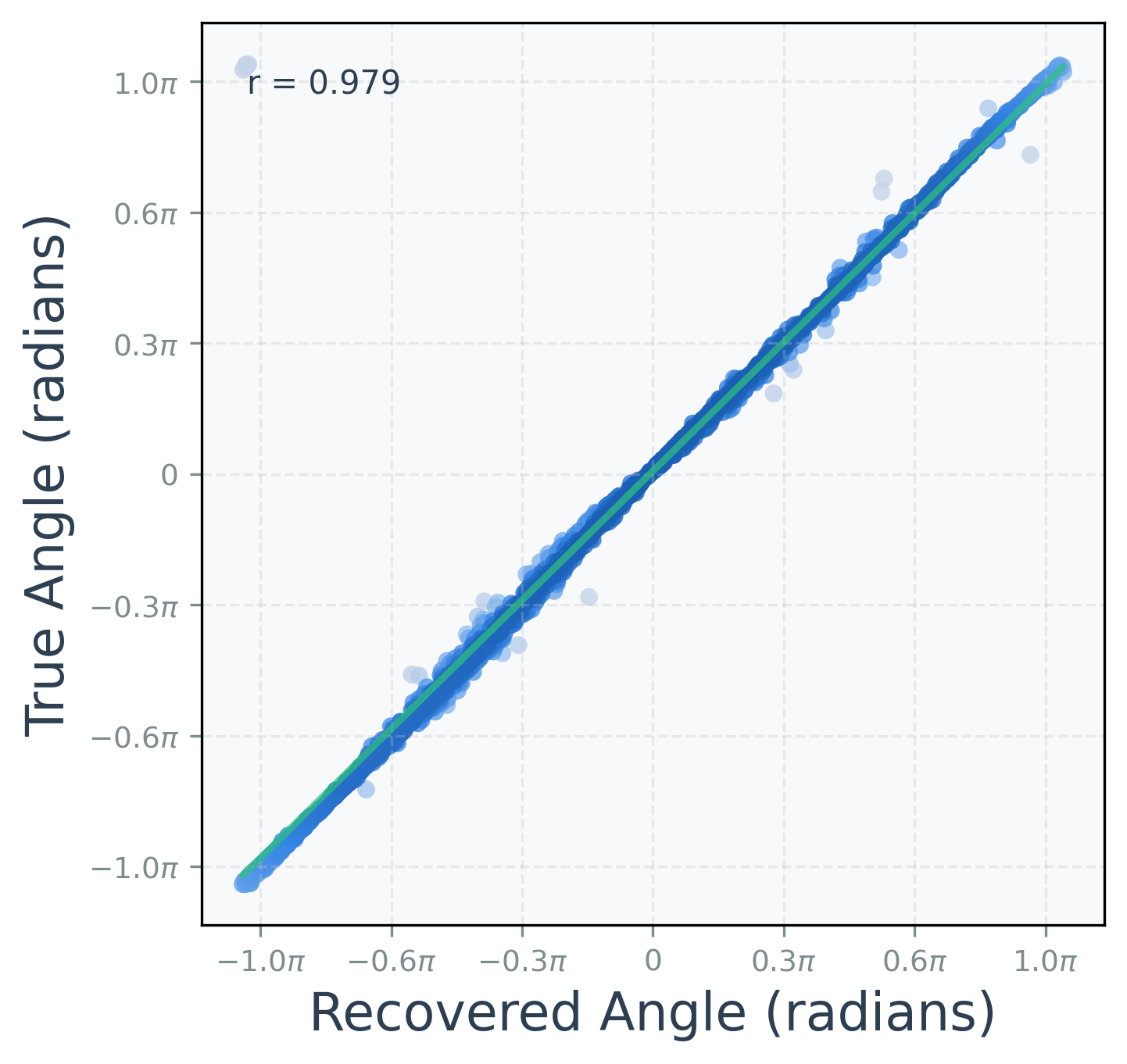}
        \caption{Poles: azimuth}
    \end{subfigure}
    \hfill
    \begin{subfigure}[b]{0.32\textwidth}
        \includegraphics[width=\textwidth]{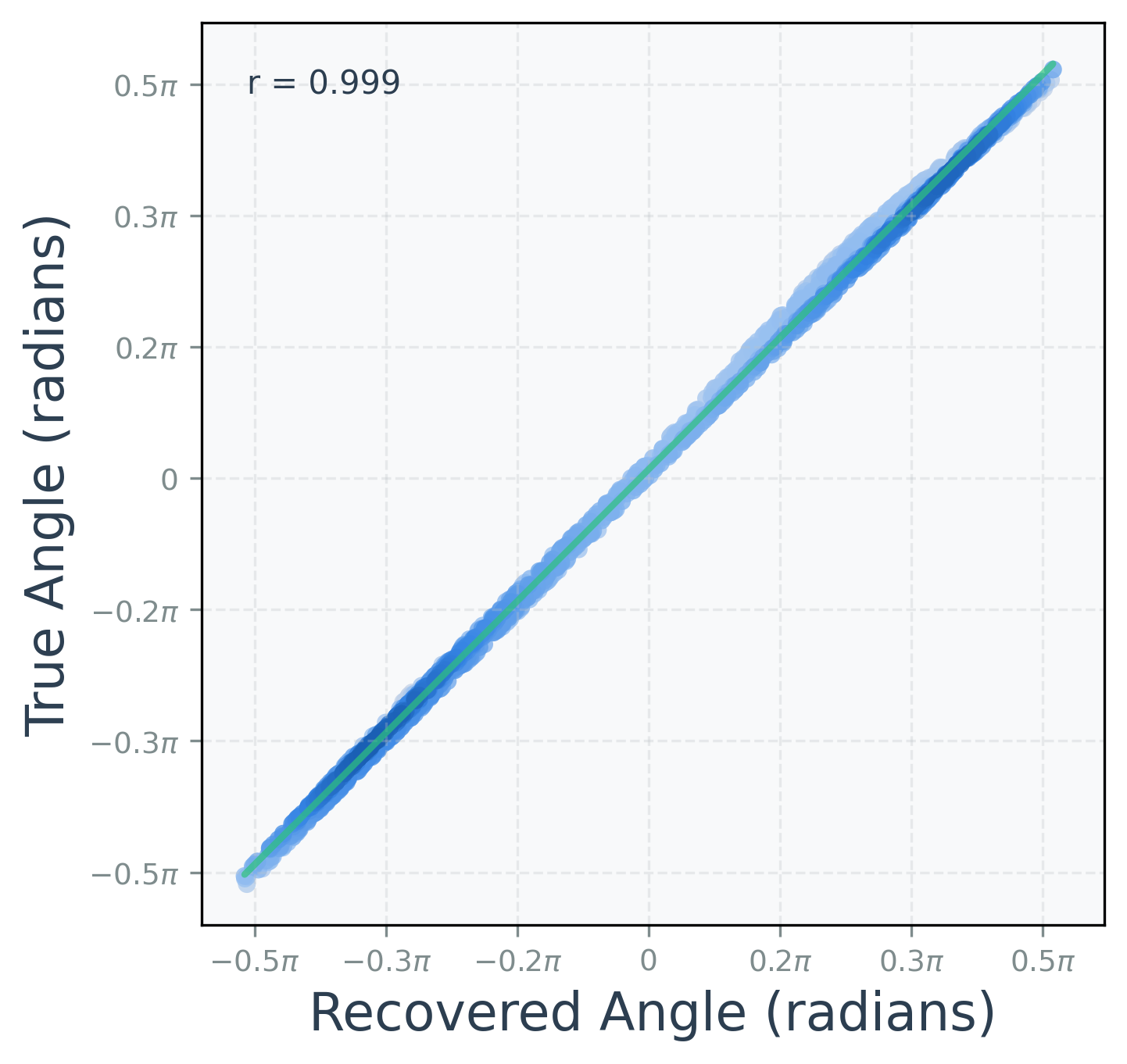}
        \caption{Poles: elevation}
    \end{subfigure}
    
    \vspace{0.5em}
    \caption{Performance of the subsampling method under three strongly
    non-uniform sampling distributions. In all cases both azimuth and elevation
    are recovered with high accuracy despite the uneven density.}
    \label{fig:strong-bias}
\end{figure}

\paragraph{Summary} Overall, these experiments validate the theoretical motivation underlying the proposed subsampling construction. Under uniform sampling, the method performs comparably to the baseline, while under increasingly non-uniform sampling it remains accurate by restoring approximate uniformity before computing spherical coordinates. In doing so, it avoids the normalization bias that affects the baseline and yields robust spherical coordinates across a wide range of sampling distributions.

\section{Conclusion}
\label{sec:conclusion}

We have developed the first density-robust construction of spherical topological coordinates by extending the subsampling-and-alignment framework for circular topological coordinates from $S^1$ to $S^2$. The principal theoretical challenge was to combine coordinates computed independently on different subsamples. We resolved this by introducing a spherical Procrustes problem together with an approximation theorem (\cref{prop:procrustes-approximation-sphere}) showing that its Euclidean relaxation remains close to the intrinsic spherical optimum, enabling consistent averaging of sphere-valued coordinates across subsamples. Our resulting construction achieves density robustness without modifying the underlying cohomological pipeline: it computes baseline spherical coordinates on approximately uniform subsamples and aligns them into a single consensus map. More broadly, these results show that density robustness in the spherical setting can be achieved through a principled geometric alignment of independently computed coordinates, rather than by modifying the underlying topological construction itself.
 
On the practical side, two implementation changes (a sparse, adjacency-based
triangle construction and just-in-time compilation) reduce the runtime of the
baseline by more than two orders of magnitude, while subsampling removes the
$\mathcal{O}(n^4)$ memory ceiling that otherwise makes degree-two computations
infeasible beyond a few hundred points. On synthetic data, the resulting
coordinates match the baseline under uniform sampling while remaining accurate under increasingly non-uniform sampling regimes in which the baseline exhibits substantial distortion. 
 
Several limitations remain. The density estimator is based on local point counts
and can be misled by outliers or isolated points, causing subsamples to
over-represent sparse regions; a more robust estimator would address this. The
baseline's zero centre-of-mass constraint can also still bias the result when the
data, though uniformly resampled, map non-uniformly onto the sphere, for
instance when the underlying shape is not a single sphere. A heuristic that
removes low-area triangles before optimization, suggested by Schonsheck and
Schonsheck~\cite{schonsheck2024spherical} and implemented in our code, mitigates
this in practice, though it lacks a theoretical justification. More fundamentally,
the method's robustness is at present empirical rather than topological: unlike
the weighting-based circular construction of Paik and Park~\cite{paik2023circular},
which is grounded in the convergence of the discrete to the continuous Laplacian,
the subsampling construction lacks an analogous analytical justification, and whether such a guarantee can be established remains open.

More broadly, our work demonstrates that the density-robust subsampling paradigm is not confined to degree-one cohomology, but extends naturally to sphere-valued coordinates arising from degree-two cohomology. This perspective provides a foundation for developing density-robust topological coordinates associated with higher-dimensional and more general target spaces. A natural next step is the setting of projective coordinates. Since $S^2\cong\mathbb{CP}^1$, our construction may be viewed as the first nontrivial instance of a broader projective framework, suggesting that the ideas developed here could ultimately unify spherical and projective topological coordinates within a common cohomological theory.

\section*{Software and Data Availability}%
An implementation of the method, together with the scripts that reproduce the
experiments and figures in this paper, is available at
\texttt{\url{https://github.com/NickNordwald/SphericalCoordinates}}.

\section*{Acknowledgments}
We wish to thank Vin de Silva, Nikolas Schonsheck, and Qiquan Wang for helpful discussions. We would also like to thank Jun Hou Fung for providing the code accompanying \cite{blumberg2024subsampling}.  
I.G.R.~has received funding from the Swiss State Secretariat for Education, Research, and Innovation (SERI). The funders had no role in the preparation of the manuscript or the decision to publish.
A.M.~is supported by the EPSRC AI Hub on Mathematical Foundations of Intelligence: An ``Erlangen Programme'' for AI No.\ EP/Y028872/1.


\printbibliography

\end{document}